\documentclass[11pt,reqno]{amsart}
\usepackage{a4wide}

\numberwithin{equation}{section}
\usepackage{mathrsfs}
\usepackage{amsfonts}
\usepackage{amsmath}
\usepackage{stmaryrd}
\usepackage{amssymb}
\usepackage{amsthm}
\usepackage{mathrsfs}
\usepackage{url}
\usepackage{amsfonts}
\usepackage{amscd}
\usepackage{indentfirst}
\usepackage{enumerate}
\usepackage{amsmath,amsfonts,amssymb,amsthm}
\usepackage{amsmath,amssymb,amsthm,amscd}
\usepackage{graphicx,mathrsfs}
\usepackage{appendix}
\usepackage{color}
 \usepackage[colorlinks, linkcolor=blue, citecolor=blue]{hyperref}

\newcommand{\R}{\mathbb{R}}

\newcommand{\dis}{\displaystyle}

\renewcommand{\theequation}{\arabic{section}.\arabic{equation}}

\newtheorem{Thm}{Theorem}[section]
\newtheorem{Lem}[Thm]{Lemma}

\newtheorem{Prop}[Thm]{Proposition}

\newtheorem{Rem}[Thm]{Remark}

\allowdisplaybreaks

\begin{document}

\title[Critical Lane-Emden System]{Existence and non-degeneracy of positive multi-bubbling  solutions to critical elliptic systems of Hamiltonian type}

\author{Qing Guo,\,\, Junyuan Liu\,\,and \,\, Shuangjie Peng}
 \address[Qing Guo]{College of Science, Minzu University of China, Beijing 100081, China} \email{guoqing0117@163.com}

\address[Junyuan Liu]{School of Mathematics and Statistics and Hubei Key Laboratory of Mathematical Sciences, Central China Normal University, Wuhan 430079, China} \email{jyliuccnu@163.com}

\address[Shuangjie Peng]{ School of Mathematics and  Statistics, Central China Normal University, Wuhan, P.R. China}
\email{ sjpeng@ccnu.edu.cn}


\date{\today}

\begin{abstract}
This paper deals with the following  critical elliptic systems of Hamiltonian type, which are variants of the critical Lane-Emden systems and analogous to the prescribed curvature problem:
\begin{equation*}
\begin{cases}
-\Delta u_1=K_1(y)u_2^{p},\  y\in \R^N,\\
-\Delta u_2=K_2(y)u_1^{q}, \  y\in \R^N,\\
u_1,u_2>0,
\end{cases}
\end{equation*}
where $N\geq 5, p,q\in(1,\infty)$ with $\frac1{p+1}+\frac1{q+1}=\frac{N-2}N$, $K_1(y)$ and $K_2(y)$ are positive radial potentials. 

At first, under suitable conditions on $K_1,K_2$ and the certain range of the exponents $p,q$, we construct an unbounded sequence of non-radial positive vector solutions, whose energy can be made arbitrarily large. Moreover, we prove a type of non-degeneracy result by use of various Pohozaev identities, which is of great interest independently. The indefinite linear operator and strongly coupled nonlinearities make the  Hamiltonian-type systems in stark contrast both to the systems of  Gradient type and to the single critical elliptic equations in the study of the prescribed curvature problems. It is  worth noting that, in  higher-dimensional cases $(N\geq5)$, there have been no results on the existence of infinitely many bubbling solutions to critical elliptic systems, either of Hamiltonian or Gradient type.

\end{abstract}

\maketitle
%\baselineskip 18p

\section{Introduction and main results}
\setcounter{equation}{0}

In this paper, we are concerned with the multiplicity of solutions and its non-degeneracy property of the following elliptic system
\begin{equation}\label{eq1}
\begin{cases}
-\Delta u_1=K_1(y)u_2^{p},\ \ \ \ \ \ y\in \R^N,\\
-\Delta u_2=K_2(y)u_1^{q}, \ \ \ \ \ \ y\in \R^N,\\
u_1,u_2>0,\ \ (u_1,u_2)\in \dot W^{2,\frac{p+1}p}(\R^N)\times\dot W^{2,\frac{q+1}q}(\R^N),
\end{cases}
\end{equation}
where $N\geq5$ and $K_1,K_2\in C(\R^N)$ are positive radial potentials, $(p,q)$ is a pair of positive numbers lying on the critical hyperbola  \begin{align}\label{pq}
\frac1{p+1}+\frac1{q+1}=\frac{N-2}N,
\end{align}
$\dot W^{2,\frac{p+1}p}(\R^N)$ and $\dot W^{2,\frac{q+1}q}(\R^N)$ are defined as the standard homogeneous Sobolev spaces.
Without loss of generality, we may assume that $p\leq\frac{N+2}{N-2}\leq q.$

The standard Lane-Emden system
\begin{equation}\label{eq1'}
\begin{cases}
-\Delta u_1= |u_2|^{p-1}u_2,\ &in\  \Omega,\\
-\Delta u_2= |u_1|^{q-1}u_1, \  &in\  \Omega,\\
u_1=u_2=0,\ &on\ \partial\Omega
\end{cases}
\end{equation}
with a smooth bounded domain  $\Omega\subset\R^N$ for $N\geq3$ and $p,q\in(0,\infty)$
is a typical Hamiltonian-type strongly coupled elliptic systems, which have been a subject of intense interest and has a rich structure.
Due to the fact that  tools  for analyzing a single equation cannot be used in a direct way to treat the systems, it seems that less attention has been paid to the existence of solutions for strongly indefinite systems and their qualitative properties.
 One of the first result about positive solutions of \eqref{eq1'} appeared in \cite{c-f-m} based on topological arguments. In \cite{f-f}, a variational argument relying on a linking theorem was used to show an existence result.
 %Then many efforts have been devoted to the variational study on such elliptic system leading to strongly indefinite functionals.
 In \cite{b-s-r}, the existence, positivity and uniqueness of ground state solutions for \eqref{eq1'} was studied.  One may also refer to  \cite{a-c-m,s} and the surveys in \cite{f}.
 It is well known that the system is strongly affected by the values of the couple $(p,q)$.
 The existence theory is associated with the critical hyperbola
 \eqref{pq},
 which was introduced by  \cite{clement-figueiredo-mitidieri}
and \cite{vorst}.
According to \cite{figueiredo-felmer,hulshof-vorst}
and \cite{bonheure},
we know that if $pq\neq1$ and
$
\frac1{p+1}+\frac1{q+1}>\frac{N-2}N,
$
then problem \eqref{eq1'} has a solution. While if the domain $\Omega$ is star-shaped and if $
\frac1{p+1}+\frac1{q+1}\leq\frac{N-2}N,$ then \eqref{eq1'} has no solutions.

\smallskip

Particularly for $\Omega=\R^N$, a positive ground state $(U,V)$ to the following system was found in \cite{lions},
\begin{align}\label{eqUV}
\begin{cases}
&\displaystyle-\Delta U=|V|^{p-1}V,\ \ in\ \ \mathbb R^N,\\
&\displaystyle-\Delta V=|U|^{q-1}U,\ \ in\ \ \mathbb R^N,\\
&\displaystyle(U,V)\in \dot{W}^{2,\frac{p+1}p}(\mathbb R^N)\times\dot{W}^{2,\frac{q+1}q}(\mathbb R^N),
\end{cases}
\end{align}
where $N\geq3$ and $(p,q)$ satisfy \eqref{pq}.
By Sobolev embeddings, there holds that
\begin{align*}
\begin{cases}
\displaystyle \dot{W}^{2,\frac{p+1}p}(\mathbb R^N)\hookrightarrow \dot{W}^{1,p^*}(\mathbb R^N)\hookrightarrow L^{q+1}(\R^N),\\
\displaystyle  \dot{W}^{2,\frac{q+1}q}(\mathbb R^N)\hookrightarrow \dot{W}^{1,q^*}(\mathbb R^N)\hookrightarrow L^{p+1}(\R^N),
\end{cases}
\end{align*}
with
$$\frac1{p^*}=\frac p{p+1}-\frac1N=\frac1{q+1}+\frac1N,\ \ \frac1{q^*}=\frac q{q+1}-\frac1N=\frac1{p+1}+\frac1N,$$
and so the following energy functional is well-defined in $\dot W^{2,\frac{p+1}p}(\R^N)\times\dot W^{2,\frac{q+1}q}(\R^N)$:
\begin{align*}
I_0(u,v):=\int_{\R^N}\nabla u\cdot\nabla v
-\frac1{p+1}\int_{\R^N}| v|^{p+1}-\frac1{q+1}\int_{\R^N}|u|^{q+1}.
\end{align*}
According to \cite{alvino-lions-trombetti},  the ground state is radially symmetric and decreasing up to a suitable translation.
Thanks to  \cite{hulshof-vorst} and \cite{wang}, the positive ground state $(U_{0,1},V_{0,1})$ of \eqref{eqUV} is unique with $U_{0,1}(0)=1$
and the family of functions
\begin{align*}
(U_{\xi,\lambda}(y),V_{\xi,\lambda}(y))=(\lambda^{\frac N{q+1}}U_{0,1}(\lambda(y-\xi)),\lambda^{\frac N{p+1}}V_{0,1}(\lambda(y-\xi)))
\end{align*}
for any $\lambda>0,\xi\in\mathbb R^N$ also solves system \eqref{eqUV}.
Sharp asymptotic behavior  of the ground states to \eqref{eqUV}  (see \cite{hulshof-vorst}) and the non-degeneracy for \eqref{eqUV}  at each ground state (see \cite{frank-kim-pistoia}) play an important role to construct bubbling solutions especially using Lyapunov-Schmidt reduction methods.

\medskip
Generally, in the literature, systems of the form
\begin{align*}
\begin{cases}
\displaystyle-\Delta u=H_v(u,v),\\
\displaystyle-\Delta v=H_u(u,v)
\end{cases}
\end{align*}
with a Hamiltonian such as $\displaystyle H(u,v)=\frac{|u|^{p+1}}{p+1}+\frac{|v|^{q+1}}{q+1}$ are usually referred to as elliptic systems of Hamiltonian type. It is also said strongly coupled
in the sense that $u\equiv0$ if and only if $v\equiv0$. Another classical system is of  Gradient type:
\begin{align*}
\begin{cases}
\displaystyle-\Delta u=F_u(u,v),\\
\displaystyle-\Delta v=F_v(u,v)
\end{cases}
\end{align*}
with a functional $\displaystyle F(u,v)=\frac1{2p}\big(\mu_1|u|^{2p}+\mu_2|v|^{2p}+2\beta|u|^{p}|v|^{p}\big)$ for example.
They are usually called nonlinear Schr\"odinger systems and  have been the subject of extensive mathematical studies in recent years, for example,
\cite{d-w-w,l-w,l-w2,liu-w,sirakov,t-v}. We also refer to
\cite{b-d-w,c-z,c-z3,g-l-w,p-w,p-p-w} for more references therein about systems with both critical and subcritical exponents.

\medskip

Hamiltonian-type systems are significantly different from the Gradient type. Due to the indefinite property of the linear operator and the strongly coupled nonlinearity of Hamiltonian-type systems, the classical variational methods can not be used directly in the study of certain problems, such as the existence of infinitely many positive solutions.
Moreover, in sharp contrast to the Gradient-type systems, it seems impossible that the strongly coupled systems  have segregated vector solutions, whose components  concentrate at different points respectively. Hence in this work, we apply finite-dimensional reduction method, combined with local Pohozaev identities, to construct infinitely many synchronized positive bubbling solutions with special symmetry. By ``synchronized" we mean that the components of bubbling solutions to \eqref{eq1} concentrate at the same  points.
It is also worth noting that, in  higher-dimensional cases $(N\geq5)$,  there are no results on the existence of infinitely many concentrated solutions to critical elliptic systems, either of Hamiltonian or Gradient type.

\medskip

System \eqref{eq1} is analogous in form to  the following scalar equation:
\begin{align}\label{eqs2}
\begin{cases}
\displaystyle-\Delta u= K(y)u^{\frac{N+2}{N-2}},\ \ u>0,\ \ y\in\R^N\\
\displaystyle u\in D^{1,2}(\R^N),
\end{cases}
\end{align}
which  has been
extensively studied, see
 \cite{a-a-p,b-c,bianchi2,b-e,c-n-y,c-y,c-l,d-n,GMPY,han,yyli1,yyli2,yyli3,yan} for example.
In particular,
Wei and Yan \cite{wei-yan-10jfa} developed a technique to apply the reduction argument for non-singularly perturbed elliptic problems and constructed  infinitely many solutions to problem \eqref{eqs2}.

\medskip
In this paper, we suppose $N\geq5$ and
 assume that $K_1$ and $K_2$ are positive and radial  satisfying  the following conditions:

{\bf (K) } There exists a constant $r_0>0$ such that for $r\in(r_0-\delta,r_0+\delta)$,
$$ K_1(r)=1-c_1|r-r_0|^{m_1}+O(|r-r_0|^{m_1+\theta_1}),$$
$$ K_2(r)=1-c_2|r-r_0|^{m_2}+O(|r-r_0|^{m_2+\theta_2}),$$
where $c_1,c_2>0,\theta_1,\theta_2>0$ are some constants, and $m_1,m_2\in[2,N-2)$. Without loss of generality, we set $m:=\min\{m_1,m_2\}=m_1$, and $m<(2p-1)(N-2)-8$.
\medskip

Moreover, to obtain sufficient decay of the bubbling solutions, we would require that the smaller exponent $p$ on the critical hyperbola \eqref{pq} should further satisfy $$\max\Big\{\frac{N+1}{N-2},\frac{N+8}{2(N-2)},\frac{N(N-2)}{(N-2)^2-(N-2-m)}\Big\}<p\leq\frac{N+2}{N-2},$$ which precisely lead to the following assumptions:

 \begin{align*}
 {\bf (P)}\ \ \ \ \ \begin{cases}
\dis\frac{13}6<p\leq\frac{7}{3},\ \ \ \ \ &if\ N=5;\vspace{2mm}\\
\max\Big\{\dis\frac{N+1}{N-2},\frac{N(N-2)}{(N-2)^2-(N-2-m)}\Big\}<p\leq\frac{N+2}{N-2},\ \ \ \ \ &if\ N\geq6.
\end{cases}
\end{align*}

Our main result in this paper can be stated as follows.
\begin{Thm}\label{th1}
Suppose  $N\geq5$ and $p,q$ satisfy \eqref{pq} and ${\bf (P)}$.
If $K_1(r),K_2(r)$ satisfy (K), then \eqref{eq1} has infinitely many non-radial positive solutions.
\end{Thm}

\begin{Rem}
We make some supplementary explanations for the assumption ${\bf (P)}$.
In fact, $p>\frac{N+8}{2(N-2)}$ is equivalent to $(2p-1)(N-2)-8>2$, which holds for $p>\frac {N+1}{N-2}$ if $N\geq6$. Thus, for $N\geq6$, $\max\Big\{\frac{N+1}{N-2},\frac{N+8}{2(N-2)},\frac{N(N-2)}{(N-2)^2-(N-2-m)}\Big\}=\max\Big\{\frac{N+1}{N-2},\frac{N(N-2)}{(N-2)^2-(N-2-m)}\Big\}$; while for $N=5$, it holds directly that $\max\{\frac{N+1}{N-2},\frac{N+8}{2(N-2)},\frac{N(N-2)}{(N-2)^2-(N-2-m)}\}=\frac{13}6$.
Moreover, note that when $m\geq3$, $\max\Big\{\frac{N+1}{N-2},\frac{N(N-2)}{(N-2)^2-(N-2-m)}\Big\}=\frac{N+1}{N-2}$.
\end{Rem}

\begin{Rem}

There are very few results about this critical elliptic systems of Hamiltonian type by use of
the finite-dimensional reduction method except \cite{KP},  where they construct families
of blowing-up solutions to some Brezis-Nirenberg-type problem on smooth bounded domains.
Different from finitely many multi-bubbling solutions studied in the bounded domain, the construction of  infinitely many multi-bubbling solutions in the whole space $\R^N$ requires relatively higher decay rate at infinity of the ground state solutions to the corresponding limit problems (see Lemma \ref{lemasym} below). Therefore, we require that the smaller exponent $p$  should not be too small ( larger than  $\frac{N+1}{N-2}$). We expect that the condition ${\bf (P)}$ is almost sharp in the present results, although  other problems involving the case $p \leq \frac{N+1}{N-2}$ would be considered in our future work.

%哈密顿型方程组与梯度型方程组有显著差别，由于前者线性算子的强不定性质和非线性强耦合性质，对于一些问题的研究，例如无穷多正解的存在性等结果，经典的变分法不能直接使用，研究困难非常大。我们将利用有限维约化结合局部Pohozaev恒等式技术构造无穷多正解。与有界区域的有限多峰解不同的是，全空间上无穷多峰解的构造对极限问题基态解的无穷远处衰减速度要求很高，于是我们要求p不能太小，这一条件在这种结果中几乎是最佳的。
\end{Rem}

\medskip
Before introducing the non-degeneracy result, we outline the main idea in the proof of Theorem \ref{th1}.
Let us fix a positive integer $k\geq k_0$, where $k_0$ is a large integer to be determined.

Set $$\mu=\mu_k=k^{\frac{N-2}{N-2-m}}$$ to be the scaling parameter. Let $2^*=\frac{2N}{N-2}$. Using the transformation $$u_1(y)\mapsto\mu^{-\frac{N}{q+1}}u_1\Big(\frac y\mu\Big),\ \ \ u_2(y)\mapsto\mu^{-\frac{N}{p+1}}u_2\Big(\frac y\mu\Big),$$ system \eqref{eq1} becomes
\begin{equation}\label{eq2}
\begin{cases}
-\Delta u_1=K_1\Big(\dis\frac y\mu\Big)u_2^{p},\ \ \ \  y\in \R^N,\vspace{0.12cm}\\
-\Delta u_2=K_2\Big(\dis\frac y\mu\Big)u_1^{q}, \ \ \ \  y\in \R^N,\vspace{0.12cm}\\
u_1,u_2>0,\ \ \ (u_1,u_2)\in \dot W^{2,\frac{p+1}p}(\R^N)\times\dot W^{2,\frac{q+1}q}(\R^N).
\end{cases}
\end{equation}

\medskip

Our proof requires the standard steps of the reduction procedure, where
  the following sharp asymptotic behavior and the non-degeneracy of the ground states for \eqref{eqUV}  play an important role.

%{hulshof-vorst} J. Hulshof, R. C. A. M. Van der Vorst, Asymptotic behavior of ground states, Proc. Am. Mat. Soc. 124 (1996) 2423-2431.

\begin{Lem}\label{lemasym}\cite{hulshof-vorst}
Assume that $p\leq\frac{N+2}{N-2}$. There exist some positive constants $a=a_{N,p}$ and $b=b_{N,p}$ depending only on $N$ and $p$ such that
\begin{align}\label{asymV}
&\lim_{r\rightarrow\infty}r^{N-2}V_{0,1}(r)=b;
\end{align}
while
\begin{align}\label{asymU}
\begin{cases}
\dis \lim_{r\rightarrow\infty}r^{(N-2)p-2}U_{0,1}(r)=a,\ \ &\text{if}\ p<\frac N{N-2};\vspace{0.12cm}\\
\dis \lim_{r\rightarrow\infty}\dis\frac{r^{N-2}}{\log r}U_{0,1}(r)=a,\ \ \ \ \ \ \ &\text{if}\ p=\frac N{N-2};\vspace{0.12cm}\\
\dis \lim_{r\rightarrow\infty}r^{N-2}U_{0,1}(r)=a,\ \ \ \ \ \ \ &\text{if}\  p>\frac N{N-2}.
\end{cases}
\end{align}
Furthermore, in the last case, we have $b^p=a((N-2)p-2)(N-(N-2)p)$.
\end{Lem}

\begin{Lem}\label{lemnonde}\cite{frank-kim-pistoia}
Set
$$
(\Psi_{0,1}^0,\Phi_{0,1}^0)=\Big(y\cdot\nabla U_{0,1}+\frac{N U_{0,1}}{q+1},y\cdot\nabla V_{0,1}+\frac{N V_{0,1}}{p+1}\Big)
$$
 and
$$(\Psi_{0,1}^l,\Phi_{0,1}^l)=(\partial_l U_{0,1},\partial_l V_{0,1}),\ \ for\ \ l=1,\ldots,N.$$
Then the space of solutions to the linear system
\begin{align}
\begin{cases}
&\displaystyle -\Delta\Psi=pV_{0,1}^{p-1}\Phi,\ \ \text{in}\ \mathbb R^N,\vspace{0.12cm}\\
&\displaystyle -\Delta\Phi=qU_{0,1}^{q-1}\Psi,\ \ \text{in}\ \mathbb R^N,\vspace{0.12cm}\\
&\displaystyle (\Psi,\Phi)\in  \dot{W}^{2,\frac{p+1}p}(\mathbb R^N)\times\dot{W}^{2,\frac{q+1}q}(\mathbb R^N)
\end{cases}
\end{align}
is spanned by
$$\Big\{(\Psi_{0,1}^0,\Phi_{0,1}^0),(\Psi_{0,1}^1,\Phi_{0,1}^1),\ldots,(\Psi_{0,1}^N,\Phi_{0,1}^N)\Big\}.$$
\end{Lem}

\medskip

\medskip

Let $y=(y',y''), y'\in\R^2, y''\in\R^{N-2}$. We define
\begin{align*}
H_s=\Big\{(u_1,u_2)\in &\dot{W}^{2,\frac{p+1}p}(\mathbb R^N)\times\dot{W}^{2,\frac{q+1}q}(\mathbb R^N), u_i\ is\ even\ in\ y_h,h=2,\ldots,N,\\
& u_i\Big(r\cos\theta,r\sin\theta,y''\Big)=u_i\Big(r\cos(\theta+\frac{2\pi j}k),r\sin(\theta+\frac{2\pi j}k),y''\Big),\ i=1,2
\Big\}.
\end{align*}
For any large integer $k>0$, let
$$x_j=\Big(r\cos\frac{2(j-1)\pi}{k},r\sin\frac{2(j-1)\pi}{k},0\Big),\ \ j=1,\ldots,k,$$
where 0 is the zero vector in $\R^{N-2}$.
Set
$$W_1(y)=W_{1,r}(y)=\sum_{j=1}^{k}U_{x_j,\lambda}(y),\ \ W_2(y)=W_{2,r}(y)=\sum_{j=1}^{k}V_{x_j,\lambda}(y).$$
Then both $W_1(y)$ and $W_2(y)$ are even in $y_h$, $h=2,\ldots,N$, and
$$
W_i(r\cos\theta,r\sin\theta,y'')=W_i\Big(r\cos(\theta+\frac{2\pi j}k),r\sin(\theta+\frac{2\pi j}k),y''\Big),
$$
where $i=1,2, y=(y',y''),y'\in\R^2,y''\in\R^{N-2}$.

\medskip

In this paper, we assume $\dis r\in\Big[r_0\mu-\frac1{\mu^{\bar\theta}},r_0\mu+\frac1{\mu^{\bar\theta}}\Big]$ for some small $\bar\theta>0$,
and $L_0\leq\lambda\leq L_1$, for some constant $L_1>L_0>0$.
We will prove Theorem \ref{th1} by verifying the following result.
\begin{Thm}\label{th2}
Under the same assumptions as in Theorem \ref{th1}, there exists an integer $k_0>0$ such that for any integer $k\geq k_0$,
\eqref{eq2} has a positive solution $(u_{1,k},u_{2,k})$ of the form
\begin{equation}\label{form1}
u_{i,k}=W_{i,r_k}(y)+\varphi_{i,k},\ \ i=1,2,
\end{equation}
where
$\dis \varphi_{i,k}\in   H_s$, $\dis r_k\in[r_0\mu-\frac1{\mu^{\bar\theta}},r_0\mu+\frac1{\mu^{\bar\theta}}]$ and as $k\rightarrow+\infty$,
$\dis \|\varphi_{i,k}\|_{L^\infty}\rightarrow0.$
\end{Thm}

\begin{Rem}
In contrast to the Gradient type systems, the strongly coupled characteristic make it more difficult to show the positivity of the solutions. We adopt some new idea to do it independently. Actually, we prove that the solutions obtained through the reduction scheme are positive vectors by establishing refined point-wise estimates and precise control of the errors (see Lemma \ref{lemdecay}).

\end{Rem}

\medskip

Recall $\dis\mu_k=k^{\frac{N-2}{N-2-m}}$. To state our non-degeneracy result we  return to the solutions $(v_{1,k},v_{2,k})$ of the original problem \eqref{eq1} by taking $$v_{1,k}(y)=\mu_k^{\frac N{q+1}}u_{1,k}(\mu_ky),\ \ \  v_{2,k}(y)=\mu_k^{\frac N{p+1}}u_{2,k}(\mu_ky).$$
We introduce the rescaled norms $\|v_1\|_{*,1}$ and $\|v_2\|_{*,2}$ as follows:
\begin{align*}
&\|v_1\|_{*,1}= \sup_{y\in\R^N}\Big(\sum_{j=1}^k\frac{\mu_k^{\frac N{q+1}}}{(1+\mu_k|y-\bar x_j|)^{\frac{N-2}2+\tau}}\Big)^{-1}|v_{1}(y)|,\\
&\|v_2\|_{*,2}= \sup_{y\in\R^N}\Big(\sum_{j=1}^k\frac{\mu_k^{\frac N{p+1}}}{(1+\mu_k|y-\bar x_j|)^{\frac{N-2}2+\tau}}\Big)^{-1}|v_{2}(y)|
\end{align*}
 with $\dis \bar x_j=\frac{x_j}{\mu_k}, \tau=1+\bar\eta$ and $\bar\eta>0$ small.
Set $$W_{1,\bar r_k}(y)=\sum_{j=1}^{k}U_{\bar x_j,\lambda_k}(y),\ \ W_{2,\bar r_k}(y)=\sum_{j=1}^{k}V_{\bar x_j,\lambda_k}(y)$$
with $ \bar r_k=|\bar x_j|$, $\lambda_k=\lambda\mu_k$.
 Theorem \ref{th2} indeed gives a solution $(v_{1,k},v_{2,k})$ to \eqref{eq1} of the form
\begin{equation}\label{vi}
v_{i,k}=W_{i,\bar r_k}(y)+\bar\varphi_{i,k},\ \ i=1,2,
\end{equation}
where
$\dis\bar\varphi_{i,k}\in   H_s$, for some $\bar\theta>0$ small, $\dis|\bar r_k-r_0|=O(\frac1{\mu^{1+\bar\theta}})$, $ \dis\mu_k=O(k^{\frac{N-2-m}{N-2}})$ and as $k\rightarrow+\infty$.
Moreover, $$\|\bar\varphi_{1,k}\|_{*,1}+\|\bar\varphi_{2,k}\|_{*,2}=O\Big(\frac1{\mu_k^{\frac m2+\bar\theta}}\Big).$$

\medskip

The linearized operator  $Q_k$ related to $(v_{1,k},v_{2,k})$ is defined by
\begin{align}\label{Qk}
Q_k(\xi_1,\xi_2)=\Big(-\Delta\xi_1-pK_1(y)v_{2,k}^{p-1}\xi_2,\ \ -\Delta\xi_2-qK_2(y)v_{1,k}^{q-1}\xi_1\Big).
\end{align}
Another result of this present paper is the following:

\begin{Thm}\label{th3}
Under the assumptions in Theorem \ref{th2}, we further suppose $K_1$ and $K_2$ satisfy that
\begin{align}\label{K+}
\Delta K_i-r(\Delta K_i+\frac12(\Delta K_i)')\neq0,\ at\ r=r_0,\ for\ i=1,2.
\end{align}
If $(\xi_1,\xi_2)\in H_s$ solves $Q_k(\xi_1,\xi_2)=0$, then there holds $(\xi_1,\xi_2)=0$.

\end{Thm}

\begin{Rem}
The study of the non-degeneracy result is inspired by Guo-Musso-Peng-Yan \cite{GMPY}, which is involved with finding new solutions
whose shape is, at main order, $$\sum_{j=1}^kU_{x_j,\mu}+\sum_{j=1}^nU_{p_j,\lambda},$$ where $k$ and $n$ are large integers,
$$x_j=\Big(r\cos\frac{2(j-1)\pi}k,r\sin\frac{2(j-1)\pi}k,0\Big),\ \ j=1,\ldots,k$$ $$p_j=\Big(0,0,t\cos\frac{2(j-1)\pi}n,t\sin\frac{2(j-1)\pi}n,0\Big),\ \ j=1,\ldots,n,$$
  $r$ and $t$ are close to $r_0$. Applying Theorem \ref{th3},
we  can also obtain new couple of concentrated solutions to \eqref{eq1} by gluing together bubbles with different concentration rates, whose concentrated points are distributed on the vertices of regular polygons on two disjoint coordinate planes.
\iffalse
We particularly point out that, very recently, in their subsequent work \cite{arxivgmpy}, a similar condition to \eqref{K+} was removed, fully relying  on the expression of the unique positive solution of the equation
$-\Delta u=u^{2^*-1}$ in $\R^N$. This idea fails in our problem, because there is no explicit  expression of the ground state to \eqref{eqUV},
which is significantly different from the single critical equations, owing to the structural feature of the Hamiltonian-type systems.
\fi
\end{Rem}

\begin{Rem}
The strong nonlinear characteristic of the Hamiltonian-type system makes our research much more difficult and nontrivial compared with the prescribed curvature problem.
Some new techniques are needed in our work. Firstly, since the exponents $p$ and $q$ vary within some appropriate range on a critical hyperbola, we perform subtle decay estimates to cover a sharp range.
Secondly, since the Hamiltonian-type systems possess indefinite linear operator,  we choose some suitable workspace, which is not a Hilbert space even,  to carry out the finite-dimensional reduction procedure.
Thirdly, due to the strongly coupled nonlinearities, it seems feeble  to use the common method to show the solutions of the system are positive vectors. Instead, we take some new ways to handle this difficulty when showing the vector-solutions positive.
Finally, we will establish various local Pohozaev identities, which really makes all the difference
in the proof of the existence and non-degeneracy of multi-bubbling solutions with some special symmetry.
 %In this process we inevitably need analyze the variation range of the connected exponents $p,q$ relevant to $s$ and $N$ in every detail and particular.

\end{Rem}

This paper is organized as follows. In section 2, we perform the linear analysis and carry out the  reduction procedure in suitable workspace with weighted maximum norm.
 In section 3, we solve the reduced finite-dimensional problem and prove Theorem \ref{th1}.
 The non-degeneracy of the multi-bubbling solutions in Theorem \ref{th3} is obtained in section 4.
 Some delicate estimates are put in the appendix.

\section{Finite-dimensional reduction}
In view of the decay properties of the asymptotic solutions described in Lemma \ref{lemasym}, in the  case of $p>\frac N{N-2}$, we set
\begin{align*}
\|u\|_*= \sup_{y\in\R^N}\Big(\sum_{j=1}^k\frac{1}{(1+|y-x_j|)^{\bar\sigma}}\Big)^{-1}|u(y)|,
\end{align*}
\begin{align*}
\|f\|_{**}= \sup_{y\in\R^N}\Big(\sum_{j=1}^k\frac{1}{(1+|y-x_j|)^{\bar\sigma+2}}\Big)^{-1}|f(y)|,
\end{align*}
where
\begin{align}\label{tau}
\bar\sigma=\frac{N-2}2+\tau
\end{align} with $\tau=1+\bar\eta$ and $\bar\eta>0$ small.
Denote also  $\|(u,v)\|_*=\|u\|_*+\|v\|_*$.

Let
\begin{align*}
Y_{j,1}=\frac{\partial U_{x_j,\lambda}}{\partial r},\ Y_{j,2}=\frac{\partial U_{x_j,\lambda}}{\partial \lambda}, \ \  \ \
Z_{j,1}=\frac{\partial V_{x_j,\lambda}}{\partial r},\ Z_{j,2}=\frac{\partial V_{x_j,\lambda}}{\partial \lambda}.
\end{align*}

\vskip 0.5cm
First we consider the linear problem: for $\dis (h_{1,k},h_{2,k})\in L^{\frac{q+1}q}(\R^N)\times L^{\frac{p+1}p}(\R^N)$
\begin{align}\label{eqlinear}
\begin{cases}
L_k(\varphi_{1,k},\varphi_{2,k})=(h_{1,k},h_{2,k})+\sum_{i=1}^2c_i\sum_{j=1}^k(pV_{x_j,\lambda}^{p-1}Z_{j,i},qU_{x_j,\lambda}^{q-1}Y_{j,i}),\vspace{0.1cm}\\
(\varphi_{1,k},\varphi_{2,k})\in H_s,\vspace{0.1cm}\\
\big\langle (pV_{x_j,\lambda}^{p-1}Z_{j,i},qU_{x_j,\lambda}^{q-1}Y_{j,i}),(\varphi_{1,k},\varphi_{2,k})\big\rangle=0,\ \ j=1,\ldots,k,\ i=1,2,
\end{cases}
\end{align}
for some number $c_i$, where $\dis\langle(u_1,u_2),\ (v_1,v_2)\rangle=\int_{\R^N}(u_1v_1+u_2v_2)$,
\begin{align}\label{L}
L_k(\varphi_{1},\varphi_{2})&=
\Big(
-\Delta\varphi_{1}-pK_1(\frac y\mu)W_2^{p-1}\varphi_{2},
-\Delta\varphi_{2}-qK_2(\frac y\mu)W_1^{q-1}\varphi_{1}
\Big).
\end{align}

\begin{Lem}\label{lemlinear}
Assume that $(\varphi_{1,k},\varphi_{2,k})$ solves \eqref{eqlinear}. If $\|(h_{1,k},h_{2,k})\|_{**}$ goes to zero as $k$ goes to infinity, so does
$\|(\varphi_{1,k},\varphi_{2,k})\|_{*}$.
\end{Lem}

\begin{proof}
By contradiction, we assume that there exist $k\rightarrow\infty$, $(h_{1,k},h_{2,k})$, $\lambda_k\in[L_1,L_2]$, $r_k\in[r_0\mu-\frac1{\mu^{\bar\theta}},r_0\mu+\frac1{\mu^{\bar\theta}}]$
and some $(\varphi_{1,k},\varphi_{2,k})$ solving \eqref{eqlinear}, with  $\|(h_{1,k},h_{2,k})\|_{**}\rightarrow0$ and
$\|(\varphi_{1,k},\varphi_{2,k})\|_{*}\geq c'>0$. We may assume that $\|(\varphi_{1,k},\varphi_{2,k})\|_{*}=1$.
For simplicity, we drop the subscript $k$ and write \eqref{eqlinear} as
\begin{align}\label{eqinte1-20220122-1}
\begin{split}
 \varphi_1(y)&=\int_{\R^N}\frac{1}{|y-z|^{N-2}}\Big(pK_1(\frac z\mu)W_2^{p-1}\varphi_2(z)+h_1(z)+\sum_{i=1}^2c_ip\sum_{j=1}^kV_{x_j,\lambda}^{p-1}Z_{j,i}\Big)dz,
\\
 \varphi_2(y)&=\int_{\R^N}\frac{1}{|y-z|^{N-2}}\Big(qK_2(\frac z\mu)W_1^{q-1}\varphi_1(z)+h_2(z)+\sum_{i=1}^2c_iq\sum_{j=1}^kU_{x_j,\lambda}^{q-1}Y_{j,i}\Big)dz.
\end{split}\end{align}

Define
\begin{align*}
\Omega_j=\left\{y:y=(y',y'')\in\R^2\times\R^{N-2},\Big\langle\frac{y'}{|y'|},\frac{x_j}{|x_j|}\Big\rangle\geq cos\frac\pi k\right\}.
\end{align*}
We only deal with the first equation of \eqref{eqinte1-20220122-1}, since  the other one can be estimated  in the same way.
From Lemma \ref{lemb2},   there exists some $\theta>0$ such that
\begin{align}\label{1-20220122-2}
&\left|\int_{\R^N}\frac{p}{|y-z|^{N-2}}K_1(\frac z\mu)W_2^{p-1}\varphi_2(z)dz\right|\nonumber \\
 \leq &C\|\varphi_2\|_*\int_{\R^N}\frac{1}{|y-z|^{N-2}}\Big(\sum_{j=1}^k\frac{1}{(1+|z-x_j|)^{N-2}}\Big)^{p-1}\sum_{j=1}^k\frac{1}{(1+|z-x_j|)^{\bar\sigma}}dz\nonumber \\
  \leq &C\|\varphi_2\|_*\sum_{j=1}^k\frac{1}{(1+|y-x_j|)^{\bar\sigma+\theta}}.
\end{align}
In fact, since $\dis p>\max\Big\{\frac {N+1}{N-2},\frac{N(N-2)}{(N-2)^2-(N-2-m)}\Big\}$,  there exists $\dis\tau_1\in\Big[\frac{N-2-m}{N-2},N-2-\frac Np\Big)$ satisfying $\dis (N-2)(p-1)-(p-1)\tau_1-\tau_1-2>0$, and then for $z\in \Omega_i$, $i=1,\ldots,k$, and any $\dis \tau_1\geq \frac{N-2-m}{N-2}$, there exists some $\theta>0$ such that
\begin{align*}
&\Big(\sum_{j=1}^k\frac{1}{(1+|z-x_j|)^{N-2}}\Big)^{p-1}\sum_{j=1}^k\frac{1}{(1+|z-x_j|)^{\bar\sigma}}\\
&\leq   C\frac{1}{(1+|z-x_i|)^{(N-2)(p-1)-(p-1)\tau_1+\frac{N-2}2+\tau-\tau_1}}
\leq \frac{C}{(1+|z-x_i|)^{\frac{N-2}2+\tau+2+\theta}}.
\end{align*}

Moreover, from Lemma \ref{lemb2}, there holds
\begin{align}\label{3}
\Big|\int_{\R^N}\frac{1}{|y-z|^{N-2}}h_1(z)dz\Big|
&\leq C\|h_1\|_{**}\int_{\R^N}\frac{1}{|y-z|^{N-2}}\sum_{j=1}^k\frac{1}{(1+|z-x_j|)^{\bar\sigma+2}}dz\nonumber \\
&\leq C\|h_1\|_{**}\sum_{j=1}^k\frac{1}{(1+|y-x_j|)^{\bar\sigma}}
\end{align}
and since $0<\bar\sigma<N-2$,
\begin{align}\label{4}
&\Big|\int_{\R^N}\frac{1}{|y-z|^{N-2}}\sum_{j=1}^kV_{x_j,\lambda}^{p-1}Z_{j,i}dz\Big|\nonumber  \\
 \leq &C \int_{\R^N}\frac{1}{|y-z|^{N-2}}\sum_{j=1}^k\frac{1}{(1+|z-x_j|)^{p(N-2)}}dz\leq C \sum_{j=1}^k\frac{1}{(1+|y-x_j|)^{\bar\sigma}}.
\end{align}

Now, we estimate $c_1$ and $c_2$. Multiply the two equations in \eqref{eqlinear} by $Z_{1,l}$ and $Y_{1,l}$ respectively, and integrate to find that
\begin{align}\label{cl}
\sum_{i=1}^2\sum_{j=1}^k\left\langle(pV_{x_j,\lambda}^{p-1}Z_{j,i},qU_{x_j,\lambda}^{q-1}Y_{j,i}),(Z_{1,l},Y_{1,l})\right\rangle c_l
=&\left\langle L_k(\varphi_1,\varphi_2),(Z_{1,l},Y_{1,l})\right\rangle\nonumber\\
&-\left\langle (h_1,h_2),(Z_{1,l},Y_{1,l})\right\rangle.
\end{align}
From Lemma \ref{lemb1} and using $\bar\sigma>1$, we get
\begin{align*}
\left|\left\langle (h_1,h_2),(Z_{1,l},Y_{1,l})\right\rangle\right|
&\leq C \|(h_1,h_2)\|_{**}\int_{\R^N}\frac{1}{(1+|z-x_1|)^{N-2}}\sum_{j=1}^k\frac{1}{(1+|z-x_j|)^{\bar\sigma+2}}dz\\
&\leq C \|(h_1,h_2)\|_{**}.
\end{align*}

On the other hand,
\begin{align*}
&\left\langle L_k(\varphi_1,\varphi_2),(Z_{1,l},Y_{1,l})\right\rangle
\\
=&\int_{\R^N}p\Big(1-K_1(\frac y\mu)\Big)W_2^{p-1}Z_{1,l}\varphi_1+p\Big(V_{x_1,\lambda}^{p-1}-(\sum_{j=1}^kV_{x_j,\lambda})^{p-1}\Big)Z_{1,l}\varphi_1dy\\
&+\int_{\R^N}q\Big(1-K_2(\frac y\mu)\Big)W_1^{q-1}Y_{1,l}\varphi_2+q\Big(U_{x_1,\lambda}^{q-1}-(\sum_{j=1}^kU_{x_j,\lambda})^{q-1}\Big)Y_{1,l}\varphi_2dy.
\end{align*}
We have the following estimate which is put in  Appendix \ref{sb}: with some $\theta>0$,
\begin{align}\label{2-5}
\left|\left\langle L_k(\varphi_1,\varphi_2),(Z_{1,l},Y_{1,l})\right\rangle\right|=O\Big(\frac1{\mu^\theta}\Big)\|(\varphi_1,\varphi_2)\|_*.
\end{align}

But, observe that, there exists some $\bar c>0$ such that
\begin{align*}
\sum_{j=1}^k\big\langle(pV_{x_j,\lambda}^{p-1}Z_{j,i},qU_{x_j,\lambda}^{q-1}Y_{j,i}),(Z_{1,l},Y_{1,l})\big\rangle
=(\bar c+o(1))\delta_{il}.
\end{align*}
Therefore, from \eqref{cl}, we get that
\begin{align}\label{6-20220122-3}
c_l=O\Big(\frac1{\mu^\theta}\|(\varphi_1,\varphi_2)\|_*+\|(h_1,h_2)\|_{**}\Big).
\end{align}

Combining \eqref{1-20220122-2}-\eqref{6-20220122-3}, we obtain that
\begin{align}\label{varphiest}
\|(\varphi_1,\varphi_2)\|_*\leq C\left(\|(h_1,h_2)\|_{**}+\frac{\sum_{j=1}^k\frac{1}{(1+|y-x_j|)^{\bar\sigma+\theta}}}{\sum_{j=1}^k\frac{1}{(1+|y-x_j|)^{\bar\sigma}}}\right).
\end{align}

Since $\|(\varphi_1,\varphi_2)\|_*=1$, we get from
\eqref{varphiest} that there exists some $R,a>0$ such that for some $j$
\begin{align}\label{7-20220122-4}
\|(\varphi_1,\varphi_2)\|_{L^\infty(B_R(x_j))}\geq a>0.
\end{align}
However, by transformation $(\bar\varphi_1(y),\bar\varphi_2(y))=(\varphi_1(y-x_j),\varphi_2(y-x_j))$ converges uniformly in any compact set to a solution
$(u,v)$ to
\begin{align}\label{eqlimit}
\begin{cases}
-\Delta u-pV_{0,\lambda}^{p-1}v=0,\\
-\Delta v-qU_{0,\lambda}^{q-1}u=0,
\end{cases}
\end{align}
for some $\lambda\in[L_1,L_2]$. Moreover, $(u,v)$ is perpendicular to the kernel of \eqref{eqlimit}. Hence, $(u,v)=(0,0)$,
which is a contradiction to \eqref{7-20220122-4}.

\end{proof}

\medskip

As a result of Lemma \ref{lemlinear}, we can prove the following result.
\begin{Prop}\label{proplinear}
There exist $k_0>0$ and some constant $C>0$, independent of $k$, such that for all $k\geq k_0$ and all $(h_1,h_2)\in L^\infty\times L^\infty(\R^N)$ satisfying the assumptions in Lemma \ref{lemlinear},
the linear problem \eqref{eqlinear} has a unique solution $(\varphi_1,\varphi_2)\equiv \mathbb L_k(h_1,h_2)$. Moreover, there hold that
\begin{align}\label{linearest}
\|\mathbb L_k(h_1,h_2)\|_*\leq C\|(h_1,h_2)\|_{**},\ \ \ \ |c_l|\leq C\|(h_1,h_2)\|_{**}.
\end{align}
\end{Prop}

\medskip
Now we consider the following nonlinear problem

\begin{align}\label{eqnonlinear0}
\begin{cases}
\dis-\Delta (W_1+\varphi_1)=K_1\Big(\frac{|y|}\mu\Big)(W_2+\varphi_2)^{p}+\sum_{t=1}^2c_tp\sum_{i=1}^kV_{x_i,\lambda}^{p-1}Z_{i,t},\  y\in \R^N,\vspace{0.12cm}\\
\dis-\Delta (W_2+\varphi_2)=K_2\Big(\frac{|y|}\mu\Big)(W_1+\varphi_1)^{q}+\sum_{t=1}^2c_tq\sum_{i=1}^kU_{x_i,\lambda}^{q-1}Y_{i,t}, \  y\in \R^N,\vspace{0.12cm}\\
\dis(\varphi_{1},\varphi_{2})\in H_s,\vspace{0.12cm}\\
\dis\big\langle (pV_{x_j,\lambda}^{p-1}Z_{j,l},qU_{x_j,\lambda}^{q-1}Y_{j,l}),(\varphi_{1},\varphi_{2})\big\rangle=0,\ \ j=1,\ldots,k,\ l=1,2.
\end{cases}
\end{align}

In this section, we are aimed to prove that
\begin{Prop}\label{propnonlinear}
There exists $k_0>0$ and some constant $C>0$, independent of $k$, such that for all $k\geq k_0$, $L_0\leq\lambda\leq L_1$,
$\dis |r-\mu r_0|\leq 1/\mu^{\bar\theta}$, with $\bar\theta>0$ is a fixed small constant,
 problem \eqref{eqnonlinear0} has a unique solution $(\varphi_1,\varphi_2)=(\varphi_1(r,\lambda),\varphi_2(r,\lambda))$ satisfying for some small constant $\theta>0$,
\begin{align}\label{linearest}
\|(\varphi_1,\varphi_2)\|_*\leq C\Big(\frac1\mu\Big)^{\frac m2+\theta},\ \ \ \ |c_l|\leq C\Big(\frac1\mu\Big)^{\frac m2+\theta}.
\end{align}

\end{Prop}

Rewrite  problem \eqref{eqnonlinear0} as
\begin{align}\label{eqnonlinear}
\begin{cases}
\dis L_k(\varphi_{1},\varphi_{2})=R_k+N_k(\varphi_{1},\varphi_{2})+\sum_{t=1}^2c_t(p\sum_{i=1}^kV_{x_i,\lambda}^{p-1}Z_{i,t},q\sum_{i=1}^kU_{x_i,\lambda}^{q-1}Y_{i,t}),\vspace{0.12cm}\\
\dis(\varphi_{1},\varphi_{2})\in H_s,\vspace{0.12cm}\\
\dis\left\langle (pV_{x_j,\lambda}^{p-1}Z_{j,l},qU_{x_j,\lambda}^{q-1}Y_{j,l}),(\varphi_{1},\varphi_{2})\right\rangle=0,\ \ j=1,\ldots,k,\ l=1,2,
\end{cases}
\end{align}
where  operator $L_k$ is defined in \eqref{L},
\begin{align}\label{N}
&N_k(\varphi_{1},\varphi_{2})=\Big(
N_{1,k}(\varphi_2),
N_{2,k}(\varphi_1)
\Big)
\end{align}
with
\begin{align*}
&N_{1,k}(\varphi_2)=K_1\Big(\frac y\mu\Big)\Big((W_2+\varphi_2)^{p}-W_2^{p}-pW_2^{p-1}\varphi_{2}\Big),\\
&N_{2,k}(\varphi_1)=K_2\Big(\frac y\mu\Big)
\Big((W_1+\varphi_1)^{q}-W_1^{q}-qW_1^{q-1}\varphi_{1}\Big),
\end{align*}
and
\begin{align}\label{R}
 R_k &=(R_{1,k},R_{2,k})=
\Big(
K_1\Big(\frac y\mu\Big)W_2^{p}-\sum_{j=1}^kV_{x_j,\lambda}^{p},
K_2\Big(\frac y\mu\Big)W_1^{q}-\sum_{j=1}^kU_{x_j,\lambda}^{q}\Big).
\end{align}

\medskip
In the following, we will use the contraction mapping theorem to show that there exists a unique solution to  problem \eqref{eqnonlinear}
in a set in which  $\|(\varphi_1,\varphi_2)\|_*$ is small. In order to do this, we first estimate $N_k(\varphi_1,\varphi_2)$ and $R_k$. Just as before, we may drop the subscript $k$ for convenience.

\begin{Lem}\label{lemN}
If $N\geq5 $, then
\begin{align*}
\|N(\varphi_1,\varphi_2)\|_{**}\leq C\|(\varphi_1,\varphi_2)\|_*^{\min\{p,2\}}.
\end{align*}

\end{Lem}

\begin{proof}
By definition of $N=N_k$ in \eqref{N}, we have
\begin{align*}
|N_1(\varphi_2)|\leq\begin{cases}
\dis C|\varphi_2|^p,\ \ &if\ 1<p\leq2;\vspace{0.12cm}\\
\dis CW_2^{p-2}\varphi_2^2+C|\varphi_2|^p,\ &if\ p>2.
\end{cases}
\end{align*}

For $1<p\leq2$, by H\"older inequalities,
\begin{align}\label{2.4-1}
\begin{split}
|N_1(\varphi_2)|&\leq\|\varphi_2\|_*^{p}\Big(\sum_{j=1}^k\frac{1}{(1+|y-x_j|)^{\bar\sigma}}\Big)^{p}\\
&\leq C\|\varphi_2\|_*^{p}\sum_{j=1}^k\frac{1}{(1+|y-x_j|)^{\bar\sigma+2}}\Big(\sum_{j=1}^k\frac{1}{(1+|y-x_j|)^{\tilde\tau}}\Big)^{p-1}\\
&\leq C\|\varphi_2\|_*^{p}\sum_{j=1}^k\frac{1}{(1+|y-x_j|)^{\bar\sigma+2}},
\end{split}
\end{align}
where $\tilde\tau=\frac{\frac{N-2}2p+\tau p-\frac{N+2}2-\tau}{p-1}\geq\frac{N-2-m}{N-2}$ and $ \tau>\max\Big\{1+\frac2{p-1}-\frac m{N-2}-\frac{N-2}2,\frac{N-2-m}{N-2}\Big\}.$

\iffalse
\begin{align*}
\sum_{j=1}^k\frac{1}{(1+|y-x_j|)^{p}}\leq C+\sum_{j=2}^k\frac{1}{|x_j-x_1|^{p}}\leq C.
\end{align*}
Indeed, we may assume,  without loss of generalization, $y\in \Omega_1$, then $|y-x_j|\geq\frac{|x_j-x_1|}2$ and  for $p\geq\frac{N-2-m}{N-2}$
\begin{align*}
&\sum_{j=1}^k\frac{1}{(1+|y-x_j|)^{p}}\leq 1+C\sum_{j=2}^k\frac{1}{|x_j-x_1|^{p}}\\
&= 1+C(\frac K\mu)^p\sum_{j=2}^k\frac{1}{j^{p}}\leq1+C\begin{cases}\frac{K^pln K}{\mu^p},\ p\geq1; \\
\frac{  K}{\mu^p},\ p\leq1,
\end{cases}\leq C,
\end{align*}
since $k=\mu^{\frac{N-2-m}{N-2}}$.
\fi

The case $p>2$ and the other term $N_{2}(\varphi_1)$  can be estimated in the same way.

\end{proof}

Now, we estimate $R_k$.
\begin{Lem}\label{lemR}
Assume that $\dis||x_1|-\mu r_0|\leq\frac1{\mu^{\bar\theta}}$, where $\bar\theta>0$ is a fixed small constant.
Then, there exists some small $\theta>0$ such that
\begin{align}\label{RK}
\|R_k\|_{**}\leq C\Big(\frac1{\mu}\Big)^{\frac m2+\theta}.
\end{align}

\end{Lem}

\begin{proof}
Since $p\leq q$ as we set, it suffices  to deal with $\|R_{1,k}\|_{**}$.

Recall by \eqref{R} that
\begin{align}\label{2.5-0}
\begin{split}
 R_{1,k}&=
K_1(\frac y\mu)W_2^{p}-\sum_{j=1}^kV_{x_j,\lambda}^{p}\\
&=K_1(\frac y\mu)\Big(W_2^{p}-\sum_{j=1}^kV_{x_j,\lambda}^{p}\Big)
+\sum_{j=1}^kV_{x_j,\lambda}^{p}\Big(K_1(\frac y\mu)-1\Big)
:=J_1+J_2.
\end{split}
\end{align}

By symmetry, we might as well assume that $y\in\Omega_1$. Then, $|y-x_j|\geq |y-x_1|$.
Therefore,
\begin{align}\label{2.5-1}
\begin{split}
|J_1|\leq & C\frac{1}{(1+|y-x_1|)^{(p-1)(N-2)}}\sum_{j=2}^k\frac{1}{(1+|y-x_j|)^{N-2}}\\
&+C\Big(\sum_{j=2}^k\frac{1}{(1+|y-x_j|)^{N-2}}\Big)^{p}.
\end{split}
\end{align}
By Lemma \ref{lemb1}, for any $0<\alpha\leq\min\{(p-1)(N-2),N-2\}$, there holds that
\begin{align}\label{2.5-2}
\begin{split}
&\frac{1}{(1+|y-x_1|)^{(p-1)(N-2)}}\sum_{j=2}^k\frac{1}{(1+|y-x_j|)^{N-2}}\\
\leq& C\frac{1}{(1+|y-x_1|)^{p(N-2)-\alpha}}\sum_{j=2}^k\frac{1}{|x_j-x_1|^{\alpha}}\leq\frac{k^\alpha}{\mu^\alpha}\frac{1}{(1+|y-x_1|)^{p(N-2)-\alpha}}.
\end{split}
\end{align}
Since $\dis p>\frac {N+1}{N-2}$, we can choose $\dis \frac{N-2}2<\alpha\leq(N-2)p-\frac{N+2}2-\tau$ to obtain that
\begin{align*}
&\frac{1}{(1+|y-x_1|)^{(p-1)(N-2)}}\sum_{j=2}^k\frac{1}{(1+|y-x_j|)^{N-2}}
\leq\frac1{\mu^{\frac m2+\theta}}\frac{1}{(1+|y-x_1|)^{\bar\sigma+2}}.
\end{align*}
Moreover, for $y\in\Omega_1$, Lemma \ref{lemb1} implies that
\begin{align*}
 &\sum_{j=2}^k\frac{1}{(1+|y-x_j|)^{N-2}}\\
 \leq&\sum_{j=2}^k\frac{1}{(1+|y-x_1|)^{\frac{N-2}{2}}}\frac{1}{(1+|y-x_j|)^{\frac{N-2}{2}}}\\
\leq&\sum_{j=2}^k\frac{1}{|x_j-x_1|^{\frac{N-2}2-[\frac1p(\frac{N+2}2+\tau)-\frac{N-2}2]}}
\frac{1}{(1+|y-x_1|)^{\frac{N-2}2+[\frac1p(\frac{N+2}2+\tau)-\frac{N-2}2]}}.
\end{align*}
Since for $\tau=1+\bar\eta$ with $\bar\eta>0$ small, $\dis p-\frac{N+2}{2(N-2)}-\frac\tau{N-2}>\frac12$, we see
\begin{align}\label{2.5-3}
\begin{split}
&\Big(\sum_{j=2}^k\frac{1}{(1+|y-x_j|)^{N-2}}\Big)^{p}\\
\leq& C(\frac k\mu)^{p(\frac{N-2}2-(\frac1p(\frac{N+2}2+\tau)-\frac{N-2}2))}\frac{1}{(1+|y-x_1|)^{\frac{N+2}2+\tau}}\\
\leq& \frac C{\mu^{m(p-\frac{N+2}{2(N-2)}-\frac\tau{N-2})}}\frac{1}{(1+|y-x_1|)^{\frac{N+2}2+\tau}}
\leq \frac C{\mu^{\frac m2+\theta}}\frac{1}{(1+|y-x_1|)^{\frac{N+2}2+\tau}}.
\end{split}\end{align}

Hence, \eqref{2.5-1}-\eqref{2.5-3} imply that
\begin{align}\label{J1}
\|J_1\|_{**}\leq\frac C{\mu^{\frac{m}{2}+\theta}}.
\end{align}

Now we estimate
\begin{align*}
&J_2=\sum_{j=1}^kV_{x_j,\lambda}^{p}(K_1(\frac y\mu)-1).
\end{align*}
First, for $y\in\Omega_1$ and $j>1$, similar to \eqref{2.5-3}, we obtain
\begin{align}\label{R1}
\begin{split}
 \left|\sum_{j=2}^kV_{x_j,\lambda}^{p}(K_1(\frac y\mu)-1)\right|
&\leq \frac C{\mu^{\frac m2+\theta}}\frac{1}{(1+|y-x_1|)^{\frac{N+2}2+\tau}}.
\end{split}\end{align}
Next, for $y\in\Omega_1$ and $||y|-\mu r_0|\geq\delta\mu$, where $\delta>0$ is a fixed constant, we have
$
||y|-|x_1||\geq\frac{\delta\mu}{2},
$
and then
\begin{align}\label{R2}
&\left|V_{x_1,\lambda}^{p}(K_1(\frac y\mu)-1)\right|
\leq \frac C{\mu^{\frac m2+\theta}}\frac{1}{(1+|y-x_1|)^{\frac{N+2}2+\tau}}.
\end{align}
While if $y\in\Omega_1$ and $||y|-\mu r_0|\leq\delta\mu$, we have that $||y|-|x_1||\leq2\delta\mu$ and
\begin{align*}
 \left|K_1(\frac y\mu)-1\right|&\leq C\left|\frac{|y|}\mu-r_0\right|^{m_1}
\leq \frac C{\mu^{m_1}}\Big(\left||y|-|x_1|\right|^{m_1}+\left||x_1|-r_0\mu\right|^{m_1}\Big)\\&
\leq \frac C{\mu^{m_1}}\left||y|-|x_1|\right|^{m_1}+ \frac C{\mu^{{m_1}+\theta}}.
\end{align*}
Since $m_1<(2p-1)(N-2)-8$, we find  $p(N-2)-\frac{N+2}2-\tau-\frac {m_1}2-\theta>0$. Thus
\begin{align*}
&\frac{\left||y|-|x_1|\right|^{m_1}}{\mu^{m_1}}\frac{1}{(1+|y-x_1|)^{p(N-2)}}\\
\leq &\frac C{\mu^{\frac {m_1}2+\theta}}\frac{\left||y|-|x_1|\right|^{\frac {m_1}2+\theta}}{(1+|y-x_1|)^{p(N-2)}}\\
\leq &\frac C{\mu^{\frac{m_1}2+\theta}}\frac{1}{(1+|y-x_1|)^{\frac{N+2}2+\tau}(1+|y-x_1|)^{p(N-2)-\frac{N+2}2-\tau-\frac {m_1}2-\theta}}\\
\leq &\frac C{\mu^{\frac {m_1}2+\theta}}\frac{1}{(1+|y-x_1|)^{\frac{N+2}2+\tau}}.
\end{align*}
Hence,
\begin{align}\label{R3}
&\left|V_{x_1,\lambda}^{p}(K_1(\frac y\mu)-1)\right|
\leq \frac C{\mu^{\frac m2+\theta}}\frac{1}{(1+|y-x_1|)^{\frac{N+2}2+\tau}},\ \ ||y|-\mu r_0|\leq\delta\mu.
\end{align}

From \eqref{R1}-\eqref{R3}, we obtain
\begin{align}\label{J2}
\|J_2\|_{**}\leq\frac C{\mu^{\frac{m}{2}+\theta}}.
\end{align}
Combing \eqref{J1} and \eqref{J2}, we conclude the proof of \eqref{RK}.

\end{proof}

Now, we are ready to prove Proposition \ref{propnonlinear}.

\medskip

\begin{proof}[
\textbf{Proof of Proposition \ref{propnonlinear}:}]
Recall that $\dis \mu=k^{\frac{N-2}{N-2-m}}$. Let
\begin{align*}
E=\Big\{(\varphi_{1},\varphi_{2})\in &(C(\R^N))^2\cap H_s,\  \|(\varphi_{1},\varphi_{2})\|_*\leq\frac1{\mu^{\frac m2}},\\
&\int_{\mathbb R^N} (pV_{x_j,\lambda}^{p-1}Z_{j,l}\varphi_{1}+qU_{x_j,\lambda}^{q-1}Y_{j,l}\varphi_{2})=0,\ j=1,\ldots,k,\ l=1,2
\Big\}.
\end{align*}
Then we only need to  solve
\begin{align*}
(\varphi_{1},\varphi_{2})=A(\varphi_{1},\varphi_{2}):=\mathbb L_k(N(\varphi_{1},\varphi_{2}))+\mathbb L_k(R_k)
\end{align*}
with $\mathbb L_k$ defined as in Proposition \ref{proplinear}.
We will prove that $A$ is a contraction map from $E$ to $E$.

\medskip
First, by   Proposition \ref{proplinear}, Lemma \ref{lemN} and  Lemma \ref{lemR},  $A$ maps $E$ to $E$ and
\begin{align*}
 \|A(\varphi_{1},\varphi_{2})\|_*&
 \leq C\|N(\varphi_{1},\varphi_{2})\|_{**}+C\|R_k\|_{**}\\
 &\leq  C\|(\varphi_{1},\varphi_{2})\|_*^{1+\theta}+\frac C{\mu^{\frac m2+\theta}}
 \leq \frac C{\mu^{\frac m2+\theta}}
 \leq \frac 1{\mu^{\frac m2}}.
\end{align*}

\medskip
Next it is obvious that
\begin{align*}
 \|A(\varphi_{1},\varphi_{2})-A(\tilde\varphi_{1},\tilde\varphi_{2})\|_*&
 \leq C\|N(\varphi_{1},\varphi_{2})-N(\tilde\varphi_{1},\tilde\varphi_{2})\|_{**}.
\end{align*}
Similar to the  proof in \eqref{2.4-1},
\begin{align*}
\begin{split}
|N_1(\varphi_2)-N_1(\tilde\varphi_2)|&\leq C(\|\varphi_2\|_*^{p-1}+\|\tilde\varphi_2\|_*^{p-1})\|\varphi_2-\tilde\varphi_2\|_*
\Big(\sum_{j=1}^k\frac{1}{(1+|y-x_j|)^{\bar\sigma}}\Big)^{p}\\
&\leq C (\|\varphi_2\|_*^{p-1}+\|\tilde\varphi_2\|_*^{p-1})\|\varphi_2-\tilde\varphi_2\|_*\sum_{j=1}^k\frac{1}{(1+|y-x_j|)^{\bar\sigma+2}}.
\end{split}
\end{align*}

Finally, $A$ is a contraction map and it follows from the contraction mapping theorem that there exists a unique $(\varphi_1,\varphi_2)\in E,$
such that $$(\varphi_1,\varphi_2)=A(\varphi_1,\varphi_2).$$
Moreover, from Proposition \ref{proplinear} % and Lemma \ref{lemR}
we get that for some $\theta>0$,
$$\|(\varphi_1,\varphi_2)\|_*\leq C(\frac1\mu)^{\frac m2+\theta}.$$
\end{proof}

\section{Proof of the existence result}

Let \begin{align*}
F(r,\lambda)=I(W_1+\varphi_1,W_2+\varphi_2),
\end{align*}
where $r=|x_1|$, $(\varphi_1,\varphi_2)\in H_s$ is obtained in Proposition \ref{propnonlinear},
and
\begin{align*}
I(u,v):=\int_{\R^N}\nabla u\cdot\nabla v
-\frac1{p+1}\int_{\R^N}K_1\Big(\frac{|y|}{\mu}\Big)| v|^{p+1}-\frac1{q+1}\int_{\R^N}K_2\Big(\frac{|y|}{\mu}\Big)|u|^{q+1}.
\end{align*}

\begin{Prop}
We have
\begin{align*}
F(r,\lambda)&=I(W_1,W_2)+O(\frac k{\mu^{m+\theta}})\\
&=k\Big(A+\frac{\bar B_1}{\lambda^{m_2}\mu^{m_2}}+\frac{\bar B_2}{\lambda^{m_1}\mu^{m_1}}+(\frac{\tilde B_2}{\lambda^{m_1-2}\mu^{m_1}}+
\frac{\tilde B_1}{\lambda^{m_2-2}\mu^{m_2}} )(\mu r_0-r)^2\\&
\quad\quad-\sum_{j=2}^k\frac{B_2}{\lambda^{N-2}|x_j-x_1|^{N-2}}\Big)
+kO\Big(\frac1{\mu^{m+\theta}}+(\frac{1}{\mu^{m_1}}+\frac{1}{\mu^{m_2}})|\mu r_0-r|^3\Big),
\end{align*}
where $\theta>0$ is a fixed constant and $\bar B_1, \bar B_2, \tilde B_1, \tilde B_2, B_2$ are positive constants.
\end{Prop}

\begin{proof}
Since
\begin{align*}
\Big\langle (I'_u(W_1+\varphi_1,W_2+\varphi_2),I'_v(W_1+\varphi_1,W_2+\varphi_2)),(\varphi_1,\varphi_2)\Big\rangle=0, \ \forall (\varphi_1,\varphi_2)\in E,
\end{align*}
 there are $t,s\in(0,1)$ such that
\begin{align*}
 &F(r,\lambda) \\
 =&I(W_1,W_2)-\frac12\langle D^2I(W_1+t\varphi_1,W_2+s\varphi_2)(\varphi_1,\varphi_2),(\varphi_1,\varphi_2)\rangle\\
 =&I(W_1,W_2)-\frac12\int_{\R^N}\Big(2\nabla\varphi_1\cdot\nabla\varphi_2-qK_2(\frac y\mu)(W_1+t\varphi_1)^{q-1}\varphi_1^2-pK_1(\frac y\mu)(W_2+s\varphi_2)^{p-1}\varphi_2^2\Big)\\
 =&I(W_1,W_2)+\frac{1}2\int_{\R^N}qK_2(\frac y\mu)((W_1+t\varphi_1)^{q-1}-W_1^{q-1})\varphi_1^2-(N_2(\varphi_1)+R_{2,k})\varphi_1\\&
+
\frac{1}2\int_{\R^N}pK_1(\frac y\mu)((W_2+s\varphi_2)^{p-1}-W_2^{p-1})\varphi_2^2
-(N_1(\varphi_2)+R_{1,k})\varphi_2.
\end{align*}

Note that
\begin{align*}
&\int_{\R^N}(N_2(\varphi_1)+R_{2,k})\varphi_1+(N_1(\varphi_2)+R_{1,k})\varphi_2\\
\leq & C(\|N_K(\varphi_1,\varphi_2)\|_{**}+\|R_k\|_{**})\|(\varphi_1,\varphi_2)\|_*
\int_{\R^N}\sum_{j=1}^k\frac{1}{(1+|y-x_j|)^{\bar\sigma}}\sum_{j=1}^k\frac{1}{(1+|y-x_j|)^{\bar\sigma+2}}\\
\leq &\frac{Ck}{\mu^{m+\theta}}.
\end{align*}
Therefore, we obtain
\begin{align*}
F(r,\lambda)&=I(W_1,W_2)+O\Big(\frac k{\mu^{m+\theta}}\Big),
\end{align*}
and the result follows from Proposition \ref{propa1}.

\end{proof}

Since $
|x_j-x_1|=2|x_1|\sin\frac{(j-1)\pi}k,\ j=2,\ldots, k,
$
we have
\begin{align*}
 \sum_{j=2}^k\frac1{|x_j-x_1|^{N-2}}& =\begin{cases}
\dis\frac2{(2|x_1|)^{N-2}}\sum_{j=2}^{\frac k2}\frac1{(\sin\frac{(j-1)\pi}k)^{N-2}}+\frac1{(2|x_1|)^{N-2}},\  if\ k\ is\ even,\vspace{2mm}\\
 \dis\frac2{(2|x_1|)^{N-2}}\sum_{j=2}^{\left[\frac k2\right]}\frac1{(\sin\frac{(j-1)\pi}k)^{N-2}},\  if\ k\ is\ odd.
\end{cases}
\end{align*}
But
 \begin{align*}
0<c'\leq\frac{\sin\frac{(j-1)\pi}k}{\frac{(j-1)\pi}k}\leq c'',\ j=2,\ldots, \Big[\frac k2\Big].
\end{align*}
Consequently, there is a constant $B_3>0$ such that
\begin{align*}
&\sum_{j=2}^k\frac1{|x_j-x_1|^{N-2}} =\frac{B_3k^{N-2}}{|x_1| ^{N-2}}+O\Big(\frac{k}{|x_1| ^{N-2}}\Big).
\end{align*}
Therefore, for some $B_4>0$,
\begin{align*}
& F(r,\lambda)\\
=&k\Big(A+\frac{\bar B_1}{\lambda^{m_2}\mu^{m_2}}+\frac{\bar B_2}{\lambda^{m_1}\mu^{m_1}}+\big(\frac{\tilde B_2}{\lambda^{m_1-2}\mu^{m_1}}+
\frac{\tilde B_1}{\lambda^{m_2-2}\mu^{m_2}}\big )(\mu r_0-r)^2-\frac{B_4k^{N-2}}{\lambda^{N-2}r^{N-2}}\Big)\\&
+kO\Big(\frac1{\mu^{m+\theta}}+\big(\frac{1}{\mu^{m_1}}+\frac{1}{\mu^{m_2}}\big)|\mu r_0-r|^3\Big),\\
& \frac{\partial F(r,\lambda)}{\partial\lambda}\\
=&k\Big(-\frac{\bar B_1m_2}{\lambda^{m_2+1}\mu^{m_2}}-\frac{\bar B_2m_1}{\lambda^{m_1+1}\mu^{m_1}}+\frac{B_4(N-2)k^{N-2}}{\lambda^{N-2}r^{N-2}}\Big)\\&+kO\Big(\frac1{\mu^{m+\theta}}+\big(\frac{1}{\mu^{m_1}}+\frac{1}{\mu^{m_2}}\big)|\mu r_0-r|^2\Big).
\end{align*}

If $m=m_1<m_2$, we set
 $\lambda_0$ to be the solution of
$\dis
-\frac{m\bar B_2}{\lambda^{m+1}}+\frac{B_4(N-2)}{\lambda^{N-1}r_0^{N-2}}=0,
$
while for the case of $m=m_1=m_2$, $\dis
-\frac{m(\bar B_2+\bar B_1)}{\lambda^{m+1}}+\frac{B_4(N-2)}{\lambda^{N-1}r_0^{N-2}}=0.
$
Then
$$
\lambda_0=\begin{cases}\Big(\dis\frac{B_4(N-2)}{m\bar B_2r_0^{N-2}}\Big)^{\frac1{N-2-m}},\ &m=m_1<m_2,\vspace{2mm}\\
\Big(\dis\frac{B_4(N-2)}{m(\bar B_2+\bar B_1)r_0^{N-2}}\Big)^{\frac1{N-2-m}},\ &m=m_1=m_2.
\end{cases}
$$

Define
\begin{align*}
D=\Big\{(r,\lambda):r\in\Big[r_0\mu-\frac1{\mu^{\bar\theta}},r_0\mu+\frac1{\mu^{\bar\theta}}\Big],\ \lambda\in
\Big[\lambda_0-\frac1{\mu^{\frac23\bar\theta}},\lambda_0+\frac1{\mu^{\frac23\bar\theta}}\Big]\Big\},
\end{align*}
where $\bar\theta>0$ is a small constant.
For any $(r,\lambda)\in D$, we have
$\dis
\frac r\mu=r_0+O(\frac1{\mu^{1+\bar\theta}}).
$
Then,
$$
r^{N-2}=\mu^{N-2}\Big(r_0^{N-2}+O(\frac1{\mu^{1+\bar\theta}})
\Big).
$$
We just deal with the case $m=m_1<m_2$, since the other one can be handled similarly.

For $(r,\lambda)\in D,$
\begin{align}\label{F}
 F(r,\lambda)
& =k\Big(A+\Big(\frac{\bar B_2}{\lambda^m}-\frac{B_4}{\lambda^{N-2}r_0^{N-2}}\Big)\frac{1}{\mu^m}+\frac{\tilde B_2}{\lambda^{m-2}\mu^m}(\mu r_0-r)^2\nonumber \\&
\quad\quad+O\Big(\frac1{\mu^{m+\theta}}+\frac{|\mu r_0-r|^3}{\mu^{m}}+\frac k{\mu^{N-2}}\Big)
\Big),
\end{align}
\begin{align}\label{partialF}
&\frac{\partial F(r,\lambda)}{\partial\lambda}=k\Big(-\Big(\frac{\bar B_2m}{\lambda^m}+\frac{B_4(N-2)}{\lambda^{N-2}r_0^{N-2}}\Big)\frac{1}{\mu^m}
+O\Big(\frac1{\mu^{m+\theta}}+\frac{|\mu r_0-r|^2}{\mu^{m}}+\frac k{\mu^{N-2}}\Big)\Big).
\end{align}

We define
\begin{align*}
&\alpha_1=k\Big(-A-\Big(\frac{\bar B_2}{\lambda_0^m}-\frac{B_4}{\lambda_0^{N-2}r_0^{N-2}}\Big)\frac{1}{\mu^m}-\frac{1}{\mu^{m+\frac52\bar\theta}}
\Big),\ \ \alpha_2=k(-A+\eta),\\
&\bar F(r,\lambda)=-F(r,\lambda), (r,\lambda)\in D,
\end{align*}
where $\eta>0$ is a small constant.

Let
\begin{align*}
\bar F^\alpha(r,\lambda)=\{(r,\lambda)\in D, \bar F(r,\lambda)\leq\alpha\}.
\end{align*}
Consider
\begin{align*}
\begin{cases}&\dis\frac{dr}{dt}=-D_r\bar F,\ \ t>0,\vspace{2mm}\\
&\dis\frac{d\lambda}{dt}=-D_\lambda\bar F,\ \ t>0,\vspace{2mm}\\
&(r,\lambda)\in F^{\alpha_2}.
\end{cases}
\end{align*}

Following the arguments used in \cite{wei-yan-10jfa}, we can obtain
\begin{Prop}\label{propflow}
The flow $(r(t),\lambda(t))$ does not leave $D$ before it reaches $F^{\alpha_1}$.
\end{Prop}

\medskip

\begin{proof}[
\textbf{Proof of Theorem \ref{th2}:}]
Define
\begin{align*}
\Lambda=\Big\{h: h(r,\lambda)=( h_1(r,\lambda), h_2(r,\lambda))\in D, (r,\lambda)\in D,  h(r,\lambda)=(r,\lambda), if\ |r-\mu r_0|=\frac1{\mu^{\bar\theta}}
\Big\}.
\end{align*}
Let
\begin{align*}
c=\inf_{h\in\Lambda}\max_{(r,\lambda)\in D}\bar F(h(r,\lambda)).
\end{align*}

Proceeding as done in \cite{wei-yan-10jfa},  we obtain that
\begin{eqnarray*}
&&\dis (i) \,\,\,\alpha_1<c<\alpha_2;\vspace{2mm}\\
&&\dis (ii)\,\, \sup_{|r-\mu r_0|=\frac1{\mu^{\bar\theta}}}\bar F(h(r,\lambda))<\alpha_1, \forall\, h\in\Lambda.
\end{eqnarray*}
Thus we conclude that $c$ is a critical value of $\bar F$.

To complete the proof of Theorem \ref{th2}, it suffices to show that solution $(u_{1,k},u_{2,k})$ of the form \eqref{form1}
 is a positive vector solution, which can be deduced by the following result.

\end{proof}

\begin{Lem}\label{lemdecay}
For any solution $(\varphi_{1,k},\varphi_{2,k})$ to \eqref{eqnonlinear0}, with $c_1=c_2=0$ and $\displaystyle
\|(\varphi_1,\varphi_2)\|_*\leq C\Big(\frac1\mu\Big)^{\frac m2+\theta}$,  there must hold further that $$
|\varphi_i(y)|\leq\frac12 W_i(y),\ \ \ i=1,2.$$

\end{Lem}

\begin{proof}
Rewrite \eqref{eqnonlinear0} with $c_1=c_2=0$ as
\begin{align}\label{eqint}
\begin{split}
&\varphi_1(y)=\int_{\R^N}\frac{1}{|y-z|^{N-2}}\Big(pK_1(\frac z\mu)W_2^{p-1}\varphi_2(z)+N_{1,k}(\varphi_2)+R_{1,k}(z)\Big)dz,
\end{split}
\end{align}
\begin{align}\label{eqint2}
\begin{split}
&\varphi_2(y)=\int_{\R^N}\frac{1}{|y-z|^{N-2}}\Big(qK_2(\frac z\mu)W_1^{q-1}\varphi_1(z)+N_{2,k}(\varphi_1)+R_{2,k}(z)\Big)dz.
\end{split}
\end{align}

\medskip
{\bf Step 1.}  Firstly, we estimate \begin{small}
\begin{align}\label{2.6-1}
\begin{split}
 R_{1,k}&=
K_1(\frac y\mu)W_2^{p}-\sum_{j=1}^kV_{x_j,\lambda}^{p}\\
&=K_1(\frac y\mu)(W_2^{p}-\sum_{j=1}^kV_{x_j,\lambda}^{p})
+\sum_{j=1}^kV_{x_j,\lambda}^{p}\Big(K_1(\frac y\mu)-1\Big)\\
&:=J_1+J_2.
\end{split}
\end{align}
\end{small}

By symmetry, we might as well assume that $y\in\Omega_1$. Then, $|y-x_j|\geq |y-x_1|$.
Therefore,
\begin{align*}
|J_1|\leq & \frac{C}{(1+|y-x_1|)^{(N-2)(p-1)}}\sum_{j=2}^k\frac{1}{(1+|y-x_j|)^{N-2}}\\
&+C\Big(\sum_{j=2}^k\frac{1}{(1+|y-x_j|)^{N-2}}\Big)^{p}.
\end{align*}
By Lemma \ref{lemb1}, for any $0<\tau_1\leq\min\{(N-2)(p-1),N-2\}$, there holds that
\begin{align*}
&\frac{1}{(1+|y-x_1|)^{(N-2)(p-1)}}\sum_{j=2}^k\frac{1}{(1+|y-x_j|)^{N-2}}\leq \frac{C}{(1+|y-x_1|)^{p(N-2)-\tau_1}}\sum_{j=2}^k\frac{1}{|x_j-x_1|^{\tau_1}}.
\end{align*}
Since $p>\frac{N+1}{N-2}$, we can choose $\frac{N-2-m}{N-2}<\tau_1<p(N-2)-N$ to obtain that  there exists some $\sigma>0,\epsilon_0>0$   small such that                                                    \begin{align*}
&\frac{1}{(1+|y-x_1|)^{(N-2)(p-1)}}\sum_{j=2}^k\frac{1}{(1+|y-x_j|)^{N-2}}\\
\leq &\frac{k^{\tau_1}}{\mu^{\tau_1}}\frac{1}{(1+|y-x_1|)^{p(N-2)-\tau_1}}\leq\frac1{\mu^{\sigma}}\frac{1}{(1+|y-x_1|)^{N+\epsilon_0}}.
\end{align*}

Similarly, for $y\in\Omega_1$, Lemma \ref{lemb1} gives that for $\frac{N-2-m}{N-2}<\tau_1<N-2-\frac{N}{p}$,
\begin{align*}
\Big(\sum_{j=2}^k\frac{1}{(1+|y-x_j|)^{N-2}}\Big)^{p}
\leq &\Big(\sum_{j=2}^k\frac{1}{|x_j-x_1|^{\tau_1}}\frac{1}{(1+|y-x_1|)^{N-2-\tau_1}}\Big)^{p}\\
&\leq \Big(\frac k\mu\Big)^{p\tau_1}\frac{1}{(1+|y-x_1|)^{(N-2-\tau_1)p}}\leq\frac1{\mu^{\sigma}}\frac{1}{(
1+|y-x_1|)^{N+\epsilon_0}}.
\end{align*}
Hence, there exists some $\sigma>0,\epsilon_0>0$   small such that
\begin{align}\label{2.6-2}
\begin{split}
|J_1|\leq\frac C{\mu^{\sigma}}\frac{1}{(1+|y-x_1|)^{N+\epsilon_0}}.
\end{split}
\end{align}

Next we estimate
\begin{align*}
&J_2=\sum_{j=1}^kV_{x_j,\lambda}^{p}\Big(K_1(\frac y\mu)-1\Big).
\end{align*}

For $y\in\Omega_1$ and $j>1$, by Lemma \ref{lemb1}, we obtain
\begin{align*}
V_{x_j,\lambda}^{p}&\leq \frac{C}{|x_j-x_1|^{\tau_1}}\frac{1}{(1+|y-x_1|)^{p(N-2)-\tau_1}},
\end{align*}
which implies that, if we choose $\frac{N-2-m}{N-2}<\tau_1<p(N-2)-N-\epsilon_0$
there holds that
\begin{align*}
\left|\sum_{j=2}^kV_{x_j,\lambda}^{p}\Big(K_1(\frac y\mu)-1\Big)\right|
\leq & C \sum_{j=2}^k\frac{1}{|x_j-x_1|^{\tau_1}}\frac{1}{(1+|y-x_1|)^{p(N-2)-\tau_1}}\\
&\leq \frac{C}{(1+|y-x_1|)^{p(N-2)-\tau_1}}\left(\frac k\mu\right)^{\tau_1}
\leq \frac C{\mu^{\sigma}}\frac{1}{(1+|y-x_1|)^{N+\epsilon_0}}.
\end{align*}

For $y\in\Omega_1$ and $||y|-\mu r_0|\geq\delta\mu$, where $\delta>0$ is a fixed constant, we have
\begin{align*}
||y|-|x_1||\geq\frac{\delta\mu}{2},
\end{align*}
and then for $\tau_1<p(N-2)-N-\epsilon_0$,
\begin{align*}
&\left|V_{x_1,\lambda}^{p}(K_1(\frac y\mu)-1)\right|
\leq \frac C{\mu^{\tau_1}} \frac{1}{(1+|y-x_1|)^{p(N-2)-\tau_1}}.
\end{align*}
While for $y\in\Omega_1$ and $||y|-\mu r_0|\leq\delta\mu$, we have that $||y|-|x_1||\leq2\delta\mu$ and
\begin{align*}
\left|K_1(\frac y\mu)-1\right|\leq & C\left|\frac{|y|}\mu-r_0\right|^{m_1}\leq \frac C{\mu^{m_1}}\left(\left||y|-|x_1|\right|^{m_1}+\left||x_1|-r_0\mu\right|^{m_1}\right)\\
\leq & \frac C{\mu^{m_1}}\left||y|-|x_1|\right|^{m_1}+ \frac C{\mu^{m_1+\theta_1}}.
\end{align*}
As a result,
\begin{align*}
&\frac{\left||y|-|x_1|\right|^{m_1}}{\mu^{m_1}}\frac{1}{(1+|y-x_1|)^{N+2}}
\leq \frac C{\mu^{m_1}}\frac{\left||y|-|x_1|\right|^{m_1-\tau_1}}{(1+|y-x_1|)^{N+\epsilon_0}}
\frac{\left||y|-|x_1|\right|^{\tau_1}}{(1+|y-x_1|)^{p(N-2)-N-\epsilon_0}}
\end{align*}
and then
\begin{align}\label{2.6-4}
\begin{split}
|J_2|\leq\frac C{\mu^{\sigma}}\frac{1}{(1+|y-x_1|)^{N+\epsilon_0}}.
\end{split}\end{align}

Finally, \eqref{2.6-1}-\eqref{2.6-4} give that there exists some $\sigma>0,\epsilon_0>0$   small such that for $y\in\Omega_1$
\begin{align*}
|R_{1,k}|\leq\frac C{\mu^{\sigma}}\frac{1}{(1+|y-x_1|)^{N+\epsilon_0}},
\end{align*}
which, by Lemma \ref{lemb2}, implies that
\begin{align}\label{R1es}
\begin{split}
&\int_{\R^N}\frac{1}{|y-z|^{N-2}}R_{1,k}(z)dz\leq \frac C{\mu^{\sigma}}\sum_{j=1}^k\frac{1}{(1+|y-x_j|)^{N-2}}.
\end{split}
\end{align}

Similarly,
\begin{align}\label{R2es}
\begin{split}
&\int_{\R^N}\frac{1}{|y-z|^{N-2}}R_{2,k}(z)dz\leq \frac C{\mu^{\sigma}}\sum_{j=1}^k\frac{1}{(1+|y-x_j|)^{N-2}}.
\end{split}
\end{align}

{\bf Step 2.}  For the nonlinearities, since $p>1$,  using H\"older inequalities, one has
\begin{align}\label{s2-1}
\begin{split}
&\int_{\R^N}\frac{K_1(\frac z\mu)}{|y-z|^{N-2}}\Big((W_2+\varphi_2)^{p}-W_2^{p}
-pW_2^{p-1}\varphi_2
\Big)\\&
\leq C\|\varphi_2\|_*^{p}\int_{\R^N}\frac{1}{|y-z|^{N-2}}\Big(\sum_{j=1}^k\frac1{(1+|z-x_j|)^{\frac{N-2}2+\tau}}\Big)^{p}\\&
\leq \frac C{\mu^\sigma}\int_{\R^N}\frac{1}{|y-z|^{N-2}}\sum_{j=1}^k\frac1{(1+|y-x_j|)^{\frac{N-2}2+2+\tau+\epsilon}}
\Big(\sum_{j=1}^k\frac1{(1+|y-x_j|)^{\frac{p(\frac{N-2}2+\tau)-(\frac{N-2}2+2+\tau+\epsilon)}{p-1}}}\Big)^{p-1}\\&
=\frac C{\mu^\sigma}\sum_{j=1}^k\frac1{(1+|y-x_j|)^{\frac{N-2}2+\tau+\epsilon}},
\end{split}
\end{align}
where we see from ${\bf(P)}$, there exists $\epsilon>0$ such that
$\displaystyle\frac{p(\frac{N-2}2+\tau)-(\frac{N-2}2+2+\tau+\epsilon)}{p-1}=\frac{N-2}2+\tau-\frac{2+\epsilon}{p-1}>\frac{N-2-m}{N-2}$.

\medskip
Similarly,  for $\epsilon>0$,
\begin{align}\label{s2-2}
\begin{split}
&\int_{\R^N}\frac{K_2(\frac z\mu)}{|y-z|^{N-2}}\Big((W_1+\varphi_1)^{q}-W_1^{q}
-qW_1^{q-1}\varphi_1
\Big)
%\leq C\|\varphi_1\|_*^{q}\int_{\R^N}\frac{1}{|y-z|^{N-2}}\Big(\sum_{j=1}^k\frac1{(1+|z-x_j|)^{\frac{N-2}2+\tau}}\Big)^{q}\\&
\leq \frac C{\mu^\sigma}\sum_{j=1}^k\frac1{(1+|y-x_j|)^{\frac{N-2}2+\tau+\epsilon}}.
\end{split}
\end{align}

\medskip

{\bf Step 3.}  The estimates for the linearities can be obtained with the same idea as Step 2:
\begin{align}\label{s2-5}
\begin{split}
&\int_{\R^N}\frac{p}{|y-z|^{N-2}}K_1(\frac z\mu)W_2^{p-1}\varphi_2(z)dz
%\leq  C\|\varphi_2\|_*\int_{\R^N}\frac{1}{|y-z|^{N-2}}\Big(\sum_{j=1}^k\frac1{(1+|z-x_j|)^{\frac{N-2}2+\tau}}\Big)^{p}\\&
\leq\frac C{\mu^\sigma}\sum_{j=1}^k\frac1{(1+|y-x_j|)^{\frac{N-2}2+\tau+\epsilon}}
\end{split}
\end{align}
and
\begin{align}\label{s2-6}
\begin{split}
&\int_{\R^N}\frac{q}{|y-z|^{N-2}}K_2(\frac z\mu)W_1^{q-1}\varphi_1(z)dz
%\leq  C\|\varphi_1\|_*\int_{\R^N}\frac{1}{|y-z|^{N-2}}\Big(\sum_{j=1}^k\frac1{(1+|z-x_j|)^{\frac{N-2}2+\tau}}\Big)^{q}\\&
\leq\frac C{\mu^\sigma}\sum_{j=1}^k\frac1{(1+|y-x_j|)^{\frac{N-2}2+\tau+\epsilon}}.
\end{split}
\end{align}

 Substituting \eqref{R1es}-\eqref{s2-6}  into \eqref{eqint} and \eqref{eqint2}, we have that for $i=1,2$ and  $\epsilon>0$,
\begin{align*}
\begin{split}
|\varphi_i(y)|\leq \frac C{\mu^\sigma}\sum_{j=1}^k\frac1{(1+|y-x_j|)^{\frac{N-2}2+\tau+\epsilon}}.
\end{split}
\end{align*}

Since $\frac{N-2}2+\tau+\epsilon>\frac{N-2}2+\tau$, applying Lemma \ref{lemb2}, we can continue this process  iteratively until we obtain
\begin{align}\label{decayvarphi}
\begin{split}
|\varphi_i(y)|\leq \frac C{\mu^\sigma}\sum_{j=1}^k\frac1{(1+|y-x_j|)^{N-2}},\ \ i=1,2.
\end{split}
\end{align}

As a consequence,
\begin{align}
\begin{split}
|\varphi_i(y)|\leq\frac12W_i(y),
\end{split}
\end{align}
which concludes  Lemma \ref{lemdecay}.
\end{proof}

\section{The non-degeneracy of the solutions}
\subsection{Pohozaev identities}
We first consider the following two systems
  \begin{equation}\label{eqv}
\begin{cases}
-\Delta v_1=K_1(y)v_2^{p},\vspace{0.12cm}\\
-\Delta v_2=K_2(y)v_1^{q},
\end{cases}
\end{equation}
 \begin{equation}\label{eqxi}
\begin{cases}
-\Delta \xi_1=pK_1(y)v_2^{p-1}\xi_2,\vspace{0.12cm}\\
-\Delta \xi_2=qK_2(y)v_1^{q-1}\xi_1.
\end{cases}
\end{equation}

Assume that $\Omega$ is a smooth domain in $\mathbb R^N$.
\begin{Lem}\label{lempoh}
It holds that
\begin{align}\label{poh1}
\begin{split}
&-\int_{\partial\Omega}\Big(\frac{\partial v_1}{\partial\nu}\frac{\partial \xi_2}{\partial y_i}+\frac{\partial v_1}{\partial y_i}\frac{\partial \xi_2}{\partial \nu}
+\frac{\partial v_2}{\partial\nu}\frac{\partial \xi_1}{\partial y_i}+\frac{\partial v_2}{\partial y_i}\frac{\partial \xi_1}{\partial \nu}\Big)
+\int_{\partial\Omega}\Big(\langle\nabla v_1,\nabla\xi_2\rangle\nu_i+\langle\nabla v_2,\nabla\xi_1\rangle\nu_i\Big)\\
&-\int_{\partial\Omega}\Big(K_1(y)v_2^p\xi_2\nu_i+K_2(y)v_1^q\xi_1\nu_i\Big)=-\int_\Omega \Big(\frac{\partial K_1(y)}{\partial y_i}v_2^p\xi_2+\frac{\partial K_2(y)}{\partial y_i}v_1^q\xi_1\Big);
\end{split}
\end{align}

\begin{align}\label{poh2}
\begin{split}
&\int_{\partial\Omega}\Big(\frac{\partial v_1}{\partial\nu}\langle\nabla \xi_2,y-x_0\rangle+\frac{\partial \xi_1}{\partial \nu}
\langle\nabla v_2,y-x_0\rangle+\frac{\partial v_2}{\partial \nu}
\langle\nabla \xi_1,y-x_0\rangle
+\frac{\partial \xi_2}{\partial \nu}
\langle  \nabla u_1,y-x_0\rangle\\&-\langle\nabla v_1,\nabla\xi_2\rangle\langle\nu,y-x_0\rangle-\langle\nabla v_2,\nabla\xi_1\rangle\langle\nu,y-x_0\rangle
\Big)\\
&+\int_{\partial\Omega}\Big(K_1(y)v_2^p\xi_2\langle\nu,y-x_0\rangle+K_2(y)v_1^q\xi_1\langle\nu,y-x_0\rangle\Big)\\&
+\int_{\partial\Omega}\Big(\frac N{p+1}\Big(\xi_2\frac{\partial v_1}{\partial \nu}+v_1\frac{\partial\xi_2}{\partial \nu}\Big)+\frac N{q+1}\Big( \xi_1\frac{\partial v_2}{\partial \nu}
+ v_2\frac{\partial\xi_1}{\partial \nu}\Big)\Big)
\\
=&\int_\Omega \Big(\langle\nabla K_1,y-x_0\rangle v_2^p\xi_2+\langle\nabla K_2,y-x_0\rangle v_1^q\xi_1\Big),
\end{split}
\end{align}
where $\nu$ is the outward unit normal of $\partial \Omega$ at $y\in\partial\Omega$, $i=1,\ldots,N$.
\end{Lem}
\begin{proof}
To show \eqref{poh1}, we have
\begin{align}\label{1}
\begin{split}
&\int_{\Omega}\Big(-\Delta v_1
\frac{\partial \xi_2}{\partial y_i}-\Delta\xi_1\frac{\partial v_2}{\partial y_i}
-\Delta v_2
\frac{\partial \xi_1}{\partial y_i}-\Delta\xi_2\frac{\partial v_1}{\partial y_i}\Big)\\
=&\int_\Omega
\Big(K_1(y)(v_2^p\frac{\partial\xi_2}{\partial y_i}+pv_2^{p-1}\xi_2\frac{\partial v_2}{\partial y_i})
+K_2(y)(v_1^q\frac{\partial\xi_1}{\partial y_i}+qv_1^{q-1}\xi_1\frac{\partial v_1}{\partial y_i})\Big).
\end{split}
\end{align}
The RHS of \eqref{1} implies
\begin{align}\label{2}
\begin{split}
&\int_\Omega
\Big(K_1(y)(v_2^p\frac{\partial\xi_2}{\partial y_i}+pv_2^{p-1}\xi_2\frac{\partial v_2}{\partial y_i})
+K_2(y)(v_1^q\frac{\partial\xi_1}{\partial y_i}+qv_1^{q-1}\xi_1\frac{\partial v_1}{\partial y_i})\Big)\\
=&\int_\Omega
\Big(K_1(y)\frac{\partial(v_2^p\xi_2)}{\partial y_i}
+K_2(y)\frac{\partial(v_1^q\xi_1)}{\partial y_i}\Big)\\
=&-\int_\Omega
\Big(v_2^p\xi_2\frac{\partial K_1(y)}{\partial y_i}
+v_1^q\xi_1\frac{\partial K_2(y)}{\partial y_i}\Big)+\int_{\partial\Omega}\Big(K_1(y)v_2^p \xi_2\nu_i+K_2(y)v_1^q\xi_1\nu_i\Big).
\end{split}
\end{align}
The LHS of \eqref{1} reads
\begin{align}\label{3}
\begin{split}
&\int_{\Omega}\Big(-\Delta v_1
\frac{\partial \xi_2}{\partial y_i}-\Delta\xi_1\frac{\partial v_2}{\partial y_i}
-\Delta v_2
\frac{\partial \xi_1}{\partial y_i}-\Delta\xi_2\frac{\partial v_1}{\partial y_i}\Big)\\
=&\int_{\partial\Omega}\Big(\langle\nabla v_1,\nabla\xi_2\rangle\nu_i+\langle\nabla v_2,\nabla\xi_1\rangle\nu_i\Big)\\
&-\int_{\partial\Omega}\Big(\frac{\partial v_1}{\partial\nu}\frac{\partial \xi_2}{\partial y_i}+\frac{\partial v_1}{\partial y_i}\frac{\partial \xi_2}{\partial \nu}
+\frac{\partial v_2}{\partial\nu}\frac{\partial \xi_1}{\partial y_i}+\frac{\partial v_2}{\partial y_i}\frac{\partial \xi_1}{\partial \nu}\Big),
\end{split}
\end{align}
which combined with \eqref{2} gives \eqref{poh1}.

Next, we prove \eqref{poh2}.
From the system \eqref{eqv}, we have that
\begin{align}\label{4}
\begin{split}
 &\int_{\Omega}\Big(-\Delta v_1\langle\nabla \xi_2,y-x_0\rangle
-\Delta\xi_1\langle\nabla v_2,y-x_0\rangle
-\Delta v_2
\langle\nabla \xi_1,y-x_0\rangle-\Delta\xi_2\langle\nabla v_1,y-x_0\rangle\Big)\\
 =& \int_{\Omega}\Big(K_1(y)(v_2^p\langle\nabla \xi_2,y-x_0\rangle+pv_2^{p-1}\xi_2\langle\nabla v_2,y-x_0\rangle)\\&\qquad+
K_2(y)(v_1^q\langle\nabla \xi_1,y-x_0\rangle+qv_1^{q-1}\xi_1\langle\nabla v_1,y-x_0\rangle)\Big).
\end{split}
\end{align}
It is easy to see that
\begin{align}\label{5}
\begin{split}
& \int_{\Omega}\Big(K_1(y)(v_2^p\langle\nabla \xi_2,y-x_0\rangle+pv_2^{p-1}\xi_2\langle\nabla v_2,y-x_0\rangle)\\&+
K_2(y)(v_1^q\langle\nabla \xi_1,y-x_0\rangle+qv_1^{q-1}\xi_1\langle\nabla v_1,y-x_0\rangle)\Big)\\
=& \int_{\Omega}\Big(K_1(y)\langle\nabla (v_2^p\xi_2),y-x_0\rangle+
K_2(y)\langle\nabla (v_1^q\xi_1),y-x_0\rangle\Big)\\
=& \int_{\partial\Omega}\Big(K_1(y)v_2^p\xi_2\langle\nu,y-x_0\rangle+K_2(y)v_1^q\xi_1\langle\nu,y-x_0\rangle\Big)\\&
\quad-\int_\Omega\Big(v_2^p\xi_2\langle\nabla K_1,y-x_0\rangle+v_1^q\xi_1\langle\nabla K_2,y-x_0\rangle\Big)
-N\int_\Omega( K_1(y)v_2^p\xi_2+K_2(y)v_1^q\xi_1).
\end{split}
\end{align}

Moreover,
\begin{align}\label{6}
\begin{split}
 &\int_{\Omega}\Big(-\Delta v_1\langle\nabla \xi_2,y-x_0\rangle
-\Delta\xi_1\langle\nabla v_2,y-x_0\rangle
-\Delta v_2
\langle\nabla \xi_1,y-x_0\rangle-\Delta\xi_2\langle\nabla v_1,y-x_0\rangle\Big)\\
 =& -\int_{\partial\Omega}\frac{\partial v_1}{\partial\nu}\langle\nabla \xi_2,y-x_0\rangle
+\int_\Omega\frac{\partial v_1}{\partial y_j}\langle\nabla \frac{\partial \xi_2}{\partial y_j},y-x_0\rangle
+\int_\Omega\langle\nabla v_1,\nabla\xi_2\rangle\\
&-\int_{\partial\Omega}\frac{\partial \xi_1}{\partial\nu}\langle\nabla v_2,y-x_0\rangle
+\int_\Omega\frac{\partial \xi_1}{\partial y_j}\langle\nabla \frac{\partial v_2}{\partial y_j},y-x_0\rangle
+\int_\Omega\langle\nabla\xi_1,\nabla v_2\rangle\\
& -\int_{\partial\Omega}\frac{\partial v_2}{\partial\nu}\langle\nabla \xi_1,y-x_0\rangle
+\int_\Omega\frac{\partial v_2}{\partial y_j}\langle\nabla \frac{\partial \xi_1}{\partial y_j},y-x_0\rangle
+\int_\Omega\langle\nabla v_2,\nabla\xi_1\rangle\\
&-\int_{\partial\Omega}\frac{\partial \xi_2}{\partial\nu}\langle\nabla v_1,y-x_0\rangle
+\int_\Omega\frac{\partial \xi_2}{\partial y_j}\langle\nabla \frac{\partial v_1}{\partial y_j},y-x_0\rangle
+\int_\Omega\langle\nabla\xi_2,\nabla v_1\rangle\\
=&-\int_{\partial\Omega}\Big(\frac{\partial v_1}{\partial\nu}\langle\nabla \xi_2,y-x_0\rangle
+\frac{\partial \xi_1}{\partial\nu}\langle\nabla v_2,y-x_0\rangle
+\frac{\partial v_2}{\partial\nu}\langle\nabla \xi_1,y-x_0\rangle
+\frac{\partial \xi_2}{\partial\nu}\langle\nabla v_1,y-x_0\rangle\\
&\quad\quad\quad\quad+\langle\nabla v_1,\nabla\xi_2\rangle\langle\nu,y-x_0\rangle+\langle\nabla v_2,\nabla\xi_1\rangle\langle\nu,y-x_0\rangle\Big)\\&
-\int_\Omega(N-2)(\langle\nabla v_1,\nabla\xi_2\rangle+\langle\nabla v_2,\nabla\xi_1\rangle).
\end{split}
\end{align}

On the other hand, we multiply the first equation of \eqref{eqv} by $\frac N{p+1}\xi_2$, the second by $\frac N{q+1}\xi_1$, and  multiply
the first equation of \eqref{eqxi} by $\frac N{p+1}v_2$, the second by $\frac N{q+1} v_1$. Adding them together and integrating on $\Omega$, it holds that
\begin{align}\label{7'}
\begin{split}
 &(\frac N{q+1}+\frac N{p+1})\int_\Omega(\langle\nabla v_1,\nabla\xi_2\rangle+\nabla v_2,\nabla\xi_1\rangle)\\
  &+\int_{\partial\Omega}\Big(\frac N{p+1}\xi_2\frac{\partial v_1}{\partial \nu}+\frac N{q+1} \xi_1\frac{\partial v_2}{\partial \nu}
+\frac N{q+1} v_2\frac{\partial\xi_1}{\partial \nu}+\frac N{p+1}v_1\frac{\partial\xi_2}{\partial \nu}\Big)\\
 =&\Big(\frac N{p+1}+\frac {pN}{p+1}\Big) \int_\Omega K_1(y)v_2^p\xi_2+\Big(\frac N{q+1}+\frac {qN}{q+1}\Big)\int_\Omega K_2(y)v_1^q\xi_1\\
 =&N\int_\Omega( K_1(y)v_2^p\xi_2+K_2(y)v_1^q\xi_1).
 \end{split}
\end{align}
Noting that $\frac N{q+1}+\frac N{p+1}=N-2$ ,  \eqref{7'} reads
\begin{align}\label{7}
\begin{split}
 &(N-2)\int_\Omega(\langle\nabla v_1,\nabla\xi_2\rangle+\nabla v_2,\nabla\xi_1\rangle)\\
  &+\int_{\partial\Omega}\Big(\frac N{p+1}\xi_2\frac{\partial v_1}{\partial \nu}+\frac N{q+1} \xi_1\frac{\partial v_2}{\partial \nu}
+\frac N{q+1} v_2\frac{\partial\xi_1}{\partial \nu}+\frac N{p+1}v_1\frac{\partial\xi_2}{\partial \nu}\Big)\\
 =&N\int_\Omega( K_1(y)v_2^p\xi_2+K_2(y)v_1^q\xi_1).
\end{split}
\end{align}
Combining \eqref{5}, \eqref{6} and \eqref{7}, we get \eqref{poh2}.

\end{proof}
\medskip

\subsection{ Some estimates on the bubbling solutions}

Recall  $\dis \bar x_j=\frac{x_j}{\mu_k}, \bar r_k=|\bar x_j|$ and $$v_{1,k}(y)=\mu_k^{\frac N{q+1}}u_{1,k}(\mu_ky),\ \ \  v_{2,k}(y)=\mu_k^{\frac N{p+1}}u_{2,k}(\mu_ky).$$
Define
\begin{align*}
&\|v_1\|_{*,1}= \sup_{y\in\R^N}\Big(\sum_{j=1}^k\frac{\mu_k^{\frac N{q+1}}}{(1+\mu_k|y-\bar x_j|)^{\frac{N-2}2+\tau}}\Big)^{-1}|v_1(y)|,\\
&\|v_2\|_{*,2}= \sup_{y\in\R^N}\Big(\sum_{j=1}^k\frac{\mu_k^{\frac N{p+1}}}{(1+\mu_k|y-\bar x_j|)^{\frac{N-2}2+\tau}}\Big)^{-1}|v_2(y)|.\\
%&\|f_1\|_{**,1}= \sup_{y\in\R^N}\Big(\sum_{j=1}^k\frac{\mu^{\frac N{p+1}}}{(1+\mu|y-\bar x_j|)^{\frac{N+2}2+\tau}}\Big)^{-1}|f_1(y)|,\\
%&\|f_2\|_{**,2}= \sup_{y\in\R^N}\Big(\sum_{j=1}^k\frac{\mu^{\frac N{q+1}}}{(1+\mu|y-\bar x_j|)^{\frac{N+2}2+\tau}}\Big)^{-1}|f_2(y)|.
\end{align*}

Theorem \ref{th1} gives that for the multi-bubbling solutions $(v_{1,k},v_{2,k})$  to system \eqref{eq1},
$$|v_{1,k}(y)|\leq C\sum_{j=1}^k\frac{\mu_k^{\frac N{q+1}}}{(1+\mu_k|y-\bar x_j|)^{\frac{N-2}2+\tau}},\ \ \ \
 |v_{2,k}(y)|\leq C\sum_{j=1}^k\frac{\mu_k^{\frac N{p+1}}}{(1+\mu_k|y-\bar x_j|)^{\frac{N-2}2+\tau}}.
$$
In fact, Lemma \ref{lemdecay} further upgrades the estimates for $(v_{1,k},v_{2,k})$ as stated in the following lemma,
which is of importance in the study of the non-degeneracy result.

\begin{Lem}\label{lemrefine}
There exists a constant $C>0$ such that for all $y\in\R^N$,
\begin{align*}
&|v_{1,k}(y)|\leq C\sum_{j=1}^k\frac{\mu_k^{\frac N{q+1}}}{(1+\mu_k|y-\bar x_j|)^{N-2}},\ \ \ \
 |v_{2,k}(y)|\leq C\sum_{j=1}^k\frac{\mu_k^{\frac N{p+1}}}{(1+\mu_k|y-\bar x_j|)^{N-2}}.
\end{align*}

\end{Lem}

\iffalse
\begin{proof}
Since $$u_{1,k}(y)=\mu_k^{-\frac N{q+1}}v_{1,k}(\mu_k^{-1}y),\ \ u_{2,k}(y)=\mu_k^{-\frac N{p+1}}v_{2,k}(\mu_k^{-1}y)$$
 satisfy \eqref{eq2},  then we have
\begin{align}\label{eqinte1}
\begin{split}
u_{1,k}(y)&=\int_{\R^N}\frac{1}{|y-z|^{N-2}}K_1\Big(\frac z{\mu_k}\Big)u_{2,k}^{p}dz,
\\
u_{2,k}(y)&=\int_{\R^N}\frac{1}{|y-z|^{N-2}}K_2\Big(\frac z{\mu_k}\Big)u_{1,k}^{p}dz.
\end{split}
\end{align}

In view of Theorem \ref{th1}, Lemma \ref{lemb2} and H\"older inequalities, it holds that there exists $\sigma>0$ such that
\begin{align}\label{4.13}
\begin{split}
|u_{1,k}(y)|&\leq C \int_{\R^N}\frac1{|z-y|^{N-2}}\Big(\sum_{j=1}^k\frac{1}{(1+ |z-x_j|)^{\frac{N+2}2+\tau}}\Big)^p\\
&\leq C \int_{\R^N}\frac1{|z-y|^{N-2}}\sum_{j=1}^k\frac{1}{(1+ |z-x_j|)^{\frac{N+2}2+\tau+\sigma}}\Big(\sum_{j=1}^k\frac{1}{(1+ |z-x_j|)^{\beta}}\Big)^{p-1}\\
&\leq C\sum_{j=1}^k\frac{1}{(1+ |y-x_j|)^{\frac{N-2}2+\tau+\sigma}},
\end{split}
\end{align}
where, just following the estimate \eqref{2.4-1}, we can take $\beta=\frac{N-2}2+\tau-\frac{2+\sigma}{p-1}>\frac{N-2-m}{N-2}$.

Similarly, we also have for $\sigma>0$,
\begin{align}\label{4.14}
\begin{split}
|u_{2,k}(y)|&\leq C \sum_{j=1}^k\frac{1}{(1+ |y-x_j|)^{\frac{N-2}2+\tau+\sigma}},
\end{split}
\end{align}

Since in \eqref{4.13} and \eqref{4.14},  $\frac{N-2}2+\tau+\sigma>\frac{N-2}2+\tau$, we can continue this process to prove the result.

\end{proof}

\fi

\medskip

\subsection{Non-degeneracy result}

To show Theorem \ref{th3}, we suppose by contradiction
that there exist $k_n\rightarrow+\infty$ such that $\|\xi_{1,n}\|_{*,1}+\|\xi_{2,n}\|_{*,2}=1$, and
by the definition \eqref{Qk},
$$Q_k(\xi_{1,n},\xi_{2,n})=0.$$

Set $$\bar \xi_{1,n}(y)=\mu_{k_n}^{-\frac N{q+1}}\xi_{1,n}(\mu_{k_n}^{-1}y+x_{k_n,1}),\ \ \ \ \bar \xi_{2,n}(y)=\mu_{k_n}^{-\frac N{p+1}}\xi_{2,n}(\mu_{k_n}^{-1}y+x_{k_n,1}).$$

\begin{Lem}
It holds that
\begin{align*}
\begin{split}
\bar \xi_{1,n}\rightarrow b_0\Psi_0+b_1\Psi_1,\ \ \bar \xi_{2,n}\rightarrow b_0\Phi_0+b_1\Phi_1,\ \ as\ n\rightarrow+\infty
\end{split}
\end{align*}
uniformly in $C^1(B_R(0))$ for any $R>0$, where $b_0,b_1$ are some constants, and
$$\Psi_0=\frac{\partial U_{0,\mu}}{\partial \mu}\Big|_{\mu=1},\ \Psi_i=\frac{\partial U_{0,1}}{\partial y_i},\ \ \
\Phi_0=\frac{\partial V_{0,\mu}}{\partial \mu}\Big|_{\mu=1},\  \Phi_i=\frac{\partial V_{0,1}}{\partial y_i},\ \ i=1,\ldots,N.$$
\end{Lem}

\begin{proof}
Since $|\bar \xi_{1,n}|,|\bar \xi_{2,n}|\leq C$, we may assume that $\bar \xi_{1,n}\rightarrow\xi_1$, $\bar \xi_{2,n}\rightarrow\xi_2$ in $C_{loc}(\R^N)$.
Then we know that $(\xi_1,\xi_2)$ satisfies the system
\begin{align*}
\begin{cases}
&\dis-\Delta\xi_1-p V_{0,1}^{p-1}\xi_2=0,\vspace{0.12cm}\\
&\dis-\Delta\xi_2-q U_{0,1}^{q-1}\xi_1=0,
\end{cases}
\end{align*}
which combined with  the non-degeneracy of $(U,V)=(U_{0,1},V_{0,1})$ from Lemma \ref{lemnonde}, implies that
$$\xi_1=\sum_{i=0}^Nb_i\Psi_i,\ \ \ \xi_2=\sum_{i=0}^Nb_i\Phi_i.$$

Since $\xi_{1,n},\xi_{2,n}$ are both even in $y_i$ for $i=2,\ldots,N$, we obtain that $b_i=0$, $i=2,\ldots,N$.

\end{proof}
 Now we decompose
 \begin{align*}
& \xi_{1,n}=b_{0,n}\mu_{k_n}\sum_{j=1}^{k_n}\frac{\partial U_{x_{k_n,j},\mu_{k_n}}}{\partial \mu_{k_n}}+b_{1,n}\mu_{k_n}^{-1}\sum_{j=1}^{k_n}\frac{\partial U_{x_{k_n,j},\mu_{k_n}}}{\partial r}+\xi_{1,n}^*,\\
 &\xi_{2,n}=b_{0,n}\mu_{k_n}\sum_{j=1}^{k_n}\frac{\partial V_{x_{k_n,j},\mu_{k_n}}}{\partial \mu_{k_n}}+b_{1,n}\mu_{k_n}^{-1}\sum_{j=1}^{k_n}\frac{\partial V_{x_{k_n,j},\mu_{k_n}}}{\partial r}+\xi_{2,n}^*,
 \end{align*}
 where $(\xi_{1,n}^*,\xi_{2,n}^*)$  satisfies that
 \begin{align*}
 &\left\langle \Big(pV_{x_{k_n,j},\mu_{k_n}}^{p-1}\frac{\partial V_{x_{k_n,j},\mu_{k_n}}}{\partial \mu_{k_n}},qU_{x_{k_n,j},\mu_{k_n}}^{q-1}\frac{\partial U_{x_{k_n,j},\mu_{k_n}}}{\partial \mu_{k_n}}\Big),(\xi_{1,n}^*,\xi_{2,n}^*)\right\rangle=0,\\
 &\left\langle \Big(pV_{x_{k_n,j},\mu_{k_n}}^{p-1}\frac{\partial V_{x_{k_n,j},\mu_{k_n}}}{\partial r},qU_{x_{k_n,j},\mu_{k_n}}^{q-1}\frac{\partial U_{x_{k_n,j},\mu_{k_n}}}{\partial r}\Big),(\xi_{1,n}^*,\xi_{2,n}^*)\right\rangle=0.
  \end{align*}

As in the proof of Proposition \ref{propnonlinear} and Lemma \ref{lemrefine}, it is standard to obtain the following two lemmas.
\begin{Lem}\label{lemxi*}
There exists a constant $C>0$ such that
\begin{align*}
&\|\xi^*_{1,n}\|_{*,1}\leq C\frac1{\mu_{k_n}^{\frac m2+\theta}},\ \ \  \|\xi^*_{2,n}\|_{*,2}\leq C\frac1{\mu_{k_n}^{\frac m2+\theta}}.
\end{align*}
\end{Lem}

\begin{Lem}
There exists a constant $C>0$ such that for all $y\in\R^N$,
\begin{align*}
&|\xi_{1,n}(y)|\leq C\sum_{j=1}^k\frac{\mu_{k_n}^{\frac N{q+1}}}{(1+\mu_{k_n}|y-\bar x_j|)^{N-2}},\\
 &|\xi_{2,n}(y)|\leq C\sum_{j=1}^k\frac{\mu_{k_n}^{\frac N{p+1}}}{(1+\mu_{k_n}|y-\bar x_j|)^{N-2}}.
\end{align*}
\end{Lem}

\medskip

To get a contradiction, we need the following result.

\begin{Lem}\label{lemloc}
For $i=1,2$, if $\Delta K_i-(\Delta K_i+\frac12(\Delta K_1)')r\neq0$ at $r=r_0$, then we have
$\bar \xi_{i,n}\rightarrow0$ uniformly in $C^1(B_R(0))$ for any $R>0$.

\end{Lem}

\begin{proof}
We  apply Lemma \ref{lempoh} in domain $\Omega_1$.
Firstly, we consider \eqref{poh1}
\begin{align}\label{4.16}
\begin{split}
&-\int_{\partial\Omega_1}\Big(\frac{\partial v_{1,k_n}}{\partial\nu}\frac{\partial \xi_{2,n}}{\partial y_1}+\frac{\partial v_{1,k_n}}{\partial y_1}\frac{\partial \xi_{2,n}}{\partial \nu}
+\frac{\partial v_{2,k_n}}{\partial\nu}\frac{\partial \xi_{1,n}}{\partial y_1}+\frac{\partial v_{2,k_n}}{\partial y_1}\frac{\partial \xi_{1,n}}{\partial \nu}\Big)\\&
+\int_{\partial\Omega_1}\Big(\langle\nabla v_{1,k_n},\nabla\xi_{2,n}\rangle\nu_1+\langle\nabla v_{2,k_n},\nabla\xi_{1,n}\rangle\nu_1\Big)\\
&-\int_{\partial\Omega_1}\Big(K_1(y)v_{2,k_n}^p\xi_{2,n}\nu_1+K_2(y)v_{1,k_n}^q\xi_{1,n}\nu_1\Big)\\
=&-\int_{\Omega_1} \Big(\frac{\partial K_1(y)}{\partial y_1}v_{2,k_n}^p\xi_{2,n}+\frac{\partial K_2(y)}{\partial y_1}v_{1,k_n}^q\xi_{1,n}\Big).
\end{split}
\end{align}
In view of the symmetry, there hold that $$\frac{\partial v_{1,k_n}}{\partial\nu}=\frac{\partial v_{2,k_n}}{\partial\nu}
=\frac{\partial \xi_{1,n}}{\partial \nu}=\frac{\partial \xi_{2,n}}{\partial \nu}=0\ \ on\ \ \partial\Omega_1.$$
Thus,
\begin{align*}
\begin{split}
&the\ LHS\ of\ \eqref{4.16}\\
=&
\int_{\partial\Omega_1}\Big(\langle\nabla v_{1,k_n},\nabla\xi_{2,n}\rangle\nu_1+\langle\nabla v_{2,k_n},\nabla\xi_{1,n}\rangle\nu_1-K_1(y)v_{2,k_n}^p\xi_{2,n}\nu_1-K_2(y)v_{1,k_n}^q\xi_{1,n}\nu_1\Big)\\
=&
-\sin\frac{\pi}{k_n}\int_{\partial\Omega_1}\Big(\langle\nabla v_{1,k_n},\nabla\xi_{2,n}\rangle+\langle\nabla v_{2,k_n},\nabla\xi_{1,n}\rangle-K_1(y)v_{2,k_n}^p\xi_{2,n}-K_2(y)v_{1,k_n}^q\xi_{1,n}\Big),
\end{split}
\end{align*}
which, combined with \eqref{4.16}, gives
\begin{align}\label{4.17}
\begin{split}
&-\int_{\Omega_1} \Big(\frac{\partial K_1(y)}{\partial y_1}v_{2,k_n}^p\xi_{2,n}+\frac{\partial K_2(y)}{\partial y_1}v_{1,k_n}^q\xi_{1,n}\Big)\\
=&
-\sin\frac{\pi}{k_n}\int_{\partial\Omega_1}\Big(\langle\nabla v_{1,k_n},\nabla\xi_{2,n}\rangle+\langle\nabla v_{2,k_n},\nabla\xi_{1,n}\rangle-K_1(y)v_{2,k_n}^p\xi_{2,n}-K_2(y)v_{1,k_n}^q\xi_{1,n}\Big).
\end{split}
\end{align}

Next, we apply \eqref{poh2} to deal with the left hand side of \eqref{4.17}.
Also by symmetry, it holds that
\begin{align}\label{4.18}
\begin{split}
&\int_{\partial\Omega_1}\Big(K_1(y)v_{2,k_n}^p\xi_{2,n}\langle\nu,y-x_{k_n,1}\rangle+K_2(y)v_{1,k_n}^q\xi_{1,n}\langle\nu,y-x_{k_n,1}\rangle\\
&\quad\quad\quad-\langle\nabla v_{1,k_n},\nabla\xi_{2,n}\rangle\langle\nu,y-x_{k_n,1}\rangle-\langle\nabla v_{2,k_n},\nabla\xi_{1,n}\rangle\langle\nu,y-x_{k_n,1}\rangle\Big)\\
=&\int_{\Omega_1} \Big(\langle\nabla K_1,y-x_{k_n,1}\rangle v_{2,k_n}^p\xi_{2,n}+\langle\nabla K_2,y-x_{k_n,1}\rangle v_{1,k_n}^q\xi_{1,n}\Big).
\end{split}
\end{align}
Note that on $\partial\Omega_1$, it holds that $\langle\nu,y\rangle=0$. Moreover, we have $\langle\nu,x_{k_n,1}\rangle=-\sin\frac{\pi}{k_n}$.
Hence, \eqref{4.18} becomes
\begin{align}\label{4.19}
\begin{split}
&\sin\frac{\pi}{k_n}\int_{\partial\Omega_1}\Big(K_1(y)v_{2,k_n}^p\xi_{2,n}+K_2(y)v_{1,k_n}^q\xi_{1,n}
-\langle\nabla v_{1,k_n},\nabla\xi_{2,n}\rangle-\langle\nabla v_{2,k_n},\nabla\xi_{1,n}\rangle\Big)\\
=&\int_{\Omega_1} \Big(\langle\nabla K_1,y-x_{k_n,1}\rangle v_{2,k_n}^p\xi_{2,n}+\langle\nabla K_2,y-x_{k_n,1}\rangle v_{1,k_n}^q\xi_{1,n}\Big),
\end{split}
\end{align}
which, combined with \eqref{4.17}, implies that
\begin{align}\label{4.20}
\begin{split}
&-\int_{\Omega_1} \Big(\frac{\partial K_1(y)}{\partial y_1}v_{2,k_n}^p\xi_{2,n}+\frac{\partial K_2(y)}{\partial y_1}v_{1,k_n}^q\xi_{1,n}\Big)\\
=&\int_{\Omega_1} \Big(\langle\nabla K_1,y-x_{k_n,1}\rangle v_{2,k_n}^p\xi_{2,n}+\langle\nabla K_2,y-x_{k_n,1}\rangle v_{1,k_n}^q\xi_{1,n}\Big).
\end{split}
\end{align}

We observe by Lemma \ref{lemxi*} that
\begin{align}\label{4.21}
\begin{split}
&\int_{\Omega_1} \Big(v_{2,k_n}^p\xi_{2,n}+v_{1,k_n}^q\xi_{1,n}\Big)\\
=&\int_{\Omega_{1,n}} \Big((\mu_{k_n}^{-\frac{N}{p+1}}v_{2,k_n}(\mu_{k_n}^{-1}y+x_{k_n,1}))^p\bar\xi_{2,n}(y)
+(\mu_{k_n}^{-\frac{N}{q+1}}v_{1,k_n}(\mu_{k_n}^{-1}y+x_{k_n,1}))^q\bar\xi_{1,n}(y)\Big)\\
=&\int_{\R^N}\Big(V^p(b_{0,n}\Phi_0+b_{1,n}\Phi_1+\mu_{k_n}^{-\frac{N}{p+1}}\xi^*_{2,n}(\mu_{k_n}^{-1}y+x_{k_n,1}))\\&
\quad\quad\quad+U^q(b_{0,n}\Psi_0+b_{1,n}\Psi_1+\mu_{k_n}^{-\frac{N}{q+1}}\xi^*_{1,n}(\mu_{k_n}^{-1}y+x_{k_n,1}))\Big)+O\Big(\frac1{\mu_{k_n}^{\frac m2+\theta}}\Big)\\
=&O(\|\xi^*_{1,n}\|_{*,1}+\|\xi^*_{2,n}\|_{*,2})+O\Big(\frac1{\mu_{k_n}^{\frac m2+\theta}}\Big)=O\Big(\frac1{\mu_{k_n}^{\frac m2+\theta}}\Big),
\end{split}
\end{align}
where $\dis\Omega_{1,n}=\{y:\mu_{k_n}^{-1}y+x_{k_n,1}\in\Omega_1\}$.

Moreover, since $\dis\nabla K_i(x_{k_n,1})=O(|x_{k_n,1}|-r_0)=O(\frac1{\mu_{k_n}^{1+\theta}})$, we obtain
\begin{align} \label{4.22}
\begin{split}
&\int_{\Omega_1} \Big(\frac{\partial K_1(y)}{\partial y_1}v_{2,k_n}^p\xi_{2,n}+\frac{\partial K_2(y)}{\partial y_1}v_{1,k_n}^q\xi_{1,n}\Big)\\
=&\int_{\Omega_1} \Big(\Big(\frac{\partial K_1(y)}{\partial y_1}-\frac{\partial K_1(x_{k_n,1})}{\partial y_1}\Big)v_{2,k_n}^p\xi_{2,n}+\Big(\frac{\partial K_2(y)}{\partial y_1}-\frac{\partial K_2(x_{k_n,1})}{\partial y_1}\Big)v_{1,k_n}^q\xi_{1,n}\Big)
+O\Big(\frac1{\mu_{k_n}^{\frac m2+1+\theta}}\Big)\\
=&\int_{\Omega_1} \Big(\Big(\Big\langle\nabla\frac{\partial K_1(x_{k_n,1})}{\partial y_1},y-x_{k_n,1}\Big\rangle+\frac12\Big\langle\nabla^2\frac{\partial K_1(x_{k_n,1})}{\partial y_1}(y-x_{k_n,1}),y-x_{k_n,1}\Big\rangle\\&\quad\quad+O(|y-x_{k_n,1}|^3)\Big)v_{2,k_n}^p\xi_{2,n}\\&\quad
+\Big(\Big\langle\nabla\frac{\partial K_2(x_{k_n,1})}{\partial y_1},y-x_{k_n,1}\Big\rangle+\frac12\Big\langle\nabla^2\frac{\partial K_2(x_{k_n,1})}{\partial y_1}(y-x_{k_n,1}),y-x_{k_n,1}\Big\rangle\\&\quad\quad+O(|y-x_{k_n,1}|^3)\Big)v_{1,k_n}^q\xi_{1,n}\Big)
+O\Big(\frac1{\mu_{k_n}^{\frac m2+1+\theta}}\Big)\\
=&\int_{\R^N}\Big\{V^p(b_{0,n}\Phi_0+b_{1,n}\Phi_1)\Big(\Big\langle\nabla\frac{\partial K_1(x_{k_n,1})}{\partial y_1},\frac y{\mu_{k_n}}\Big\rangle+\frac12\Big\langle\nabla^2\frac{\partial K_1(x_{k_n,1})}{\partial y_1}\frac y{\mu_{k_n}},\frac y{\mu_{k_n}}\Big\rangle\Big)
\\&\quad\quad\quad+U^q(b_{0,n}\Psi_0+b_{1,n}\Psi_1)\Big(\Big\langle\nabla\frac{\partial K_2(x_{k_n,1})}{\partial y_1},\frac y{\mu_{k_n}}\rangle+\frac12\Big\langle\nabla^2\frac{\partial K_2(x_{k_n,1})}{\partial y_1}\frac y{\mu_{k_n}},\frac y{\mu_{k_n}}\Big\rangle\Big)\Big\}\\&\quad+O\Big(\frac1{\mu_{k_n}^{\frac m2+1+\theta}}+\frac1{\mu_{k_n}^3}\Big)\\
=&\frac{K_2''(x_{k_n})}{\mu_{k_n}}b_{1,n}\int_{\R^N}V^p\Phi_1y_1+\frac{\frac{\partial\Delta K_2(x_{k_n})}{\partial y_1}}{2N\mu_{k_n}^2}b_{0,n}\int_{\R^N}V^p\Phi_0|y|^2\\&\quad
+\frac{K_1''(x_{k_n})}{\mu_{k_n}}b_{1,n}\int_{\R^N}U^q\Psi_1y_1+\frac{\frac{\partial\Delta K_1(x_{k_n})}{\partial y_1}}{2N\mu_{k_n}^2}b_{0,n}\int_{\R^N}U^q\Psi_0|y|^2
+O\Big(\frac1{\mu_{k_n}^{\frac m2+1+\theta}}+\frac1{\mu_{k_n}^3}\Big).
\end{split}
\end{align}

On the other hand,  since as in the proof of \eqref{4.21},
\begin{align*}
\begin{split}
&\int_{\Omega_1} \Big(\langle\nabla K_1(x_{k_n,1}),y-x_{k_n,1}\rangle v_{2,k_n}^p\xi_{2,n}+\langle\nabla K_2(x_{k_n,1}),y-x_{k_n,1}\rangle v_{1,k_n}^q\xi_{1,n}\Big)\\&
=O\Big(\frac1{\mu_{k_n}^{\frac m2+1+\theta}}\Big),
\end{split}
\end{align*}
thus,
\begin{align}\label{4.23}
\begin{split}
&\int_{\Omega_1} \Big(\langle\nabla K_1(y),y-x_{k_n,1}\rangle v_{2,k_n}^p\xi_{2,n}+\langle\nabla K_2(y),y-x_{k_n,1}\rangle v_{1,k_n}^q\xi_{1,n}\Big)\\
=&\int_{\Omega_1} \Big(\langle\nabla K_1(y)-\nabla K_1(x_{k_n,1}),y-x_{k_n,1}\rangle v_{2,k_n}^p\xi_{2,n}\\
&\quad+\langle\nabla K_2(y)-\nabla K_2(x_{k_n,1}),y-x_{k_n,1}\rangle v_{1,k_n}^q\xi_{1,n}\Big)+O\Big(\frac1{\mu_{k_n}^{\frac m2+1+\theta}}\Big)\\
=&\int_{\Omega_1} \Big(\langle\nabla^2 K_1(x_{k_n,1})(y-x_{k_n,1}),y-x_{k_n,1}\rangle v_{2,k_n}^p\xi_{2,n}\\
&\quad+\langle\nabla^2 K_2(x_{k_n,1})(y-x_{k_n,1}),y-x_{k_n,1}\rangle v_{1,k_n}^q\xi_{1,n}\Big)
+O\Big(\frac1{\mu_{k_n}^{\frac m2+1+\theta}}+\frac1{\mu_{k_n}^3}\Big)\\
=&\int_{\R^N}V^p(b_{0,n}\Phi_0+b_{1,n}\Phi_1)\Big\langle\nabla^2 K_1(x_{k_n,1})\frac y{\mu_{k_n}},\frac y{\mu_{k_n}}\Big\rangle\\&\quad
+U^q(b_{0,n}\Psi_0+b_{1,n}\Psi_1)\Big\langle\nabla^2 K_2(x_{k_n,1})\frac y{\mu_{k_n}},\frac y{\mu_{k_n}}\Big\rangle
+O\Big(\frac1{\mu_{k_n}^{\frac m2+1+\theta}}+\frac1{\mu_{k_n}^3}\Big)\\
=&\frac{\Delta K_2(x_{k_n})}{N\mu_{k_n}^2}b_{0,n}\int_{\R^N}V^p\Phi_0|y|^2
+\frac{\Delta K_1(x_{k_n})}{N\mu_{k_n}^2}b_{0,n}\int_{\R^N}U^q\Psi_0|y|^2\\
&+O\Big(\frac1{\mu_{k_n}^{\frac m2+1+\theta}}+\frac1{\mu_{k_n}^3}\Big).
\end{split}
\end{align}

Combining \eqref{4.20}, \eqref{4.22} and \eqref{4.23} together, we obtain that
\begin{align}\label{4.24}
\begin{split}
&b_{1,n}\\
=&-\frac{b_{0,n}}{\mu_{k_n}}\frac{1}{K_1''(x_{k_n,1})\dis\int_{\R^N}U^q\Psi_1y_1+K_2''(x_{k_n,1})\int_{\R^N}V^p\Phi_1y_1}\times\\
&\Big\{\Big(\frac{\Delta K_2(x_{k_n,1})}{N}+\frac{\frac{\partial\Delta K_2(x_{k_n,1})}{\partial y_1}}{2N}\Big)\int_{\R^N}V^p\Phi_0|y|^2\\
&\quad\quad+\Big(\frac{\Delta K_1(x_{k_n,1})}{N}+\frac{\frac{\partial\Delta K_1(x_{k_n,1})}{\partial y_1}}{2N}\Big)\int_{\R^N}U^q\Psi_0|y|^2
\Big\}\\&+O\Big(\frac1{\mu_{k_n}^{\frac m2+\theta}}+\frac1{\mu_{k_n}^2}\Big).
\end{split}
\end{align}

Next, from \eqref{eq2},
\begin{align*}
\begin{split}
&\int_{\R^N} \Big(\langle\nabla K_1(y),y\rangle v_{2,k_n}^p\xi_{2,n}+\langle\nabla K_2(y),y\rangle v_{1,k_n}^q\xi_{1,n}\Big)=0
\end{split}
\end{align*}
and so
\begin{align}\label{4.25}
\begin{split}
&\int_{\Omega_1} \Big(\langle\nabla K_1(y),y\rangle v_{2,k_n}^p\xi_{2,n}+\langle\nabla K_2(y),y\rangle v_{1,k_n}^q\xi_{1,n}\Big)=0.
\end{split}
\end{align}

On the other hand, just as proved in \eqref{4.21}, we obtain
\begin{align}\label{4.26}
\begin{split}
&\quad\int_{\Omega_1} \Big(\langle\nabla K_1(x_{k_n,1}),y\rangle v_{2,k_n}^p\xi_{2,n}+\langle\nabla K_2(x_{k_n,1}),y\rangle v_{1,k_n}^q\xi_{1,n}\Big)\\
&=\int_{\Omega_1} \Big(\langle\nabla K_1(x_{k_n,1}),y-x_{k_n,1}\rangle v_{2,k_n}^p\xi_{2,n}+\langle\nabla K_2(x_{k_n,1}),y-x_{k_n,1}\rangle v_{1,k_n}^q\xi_{1,n}\Big)\\
&\quad +\int_{\Omega_1} \Big(\langle\nabla K_1(x_{k_n,1}),x_{k_n,1}\rangle v_{2,k_n}^p\xi_{2,n}+\langle\nabla K_2(x_{k_n,1}),x_{k_n,1}\rangle v_{1,k_n}^q\xi_{1,n}\Big)\\&
=O\Big(\frac1{\mu_{k_n}^{\frac m2+1+\theta}}\Big).
\end{split}
\end{align}

From \eqref{4.24},
\begin{align}\label{4.27}
\begin{split}
&\int_{\Omega_1} \Big(\langle\nabla K_1(y),y\rangle v_{2,k_n}^p\xi_{2,n}+\langle\nabla K_2(y),y\rangle v_{1,k_n}^q\xi_{1,n}\Big)\\
=&\int_{\Omega_1} \Big(\langle\nabla K_1(y)-\nabla K_1(x_{k_n,1}),y\rangle v_{2,k_n}^p\xi_{2,n}+\langle\nabla K_2(y)-\nabla K_2(x_{k_n,1}),y\rangle v_{1,k_n}^q\xi_{1,n}\Big)\\
&+O\Big(\frac1{\mu_{k_n}^{\frac m2+1+\theta}}\Big)\\
=&\int_{\R^N}\Big(V^p(b_{0,n}\Phi_0+b_{1,n}\Phi_1)\Big\langle\nabla^2 K_1(x_{k_n,1})\frac y{\mu_{k_n}},\frac y{\mu_{k_n}}+x_{k_n,1}\Big\rangle\\&
\quad\quad+U^q(b_{0,n}\Psi_0+b_{1,n}\Psi_1)\Big\langle\nabla^2K_2(x_{k_n,1})\frac y{\mu_{k_n}},\frac y{\mu_{k_n}}+x_{k_n,1}\Big\rangle\Big)
+O\Big(\frac1{\mu_{k_n}^{\frac m2+1+\theta}}+\frac1{\mu_{k_n}^3}\Big)\\
=&\frac{\Delta K_1(x_{k_n})}{N\mu_{k_n}^2}b_{0,n}\int_{\R^N}V^p\Phi_0|y|^2+\frac{b_{1,n}}{\mu_{k_n}}\int_{\R^N}V^p\Phi_1\langle\nabla^2 K_1(x_{k_n,1})y,x_{k_n,1}\rangle\\
&+\frac{\Delta K_2(x_{k_n})}{N\mu_{k_n}^2}b_{0,n}\int_{\R^N}U^q\Psi_0|y|^2+\frac{b_{1,n}}{\mu_{k_n}}\int_{\R^N}U^q\Psi_1\langle\nabla^2 K_2(x_{k_n,1})y,x_{k_n,1}\rangle\\
&+O\Big(\frac1{\mu_{k_n}^{\frac m2+1+\theta}}+\frac1{\mu_{k_n}^3}\Big)\\
=&\frac{\Delta K_1(x_{k_n})}{N\mu_{k_n}^2}b_{0,n}\int_{\R^N}V^p\Phi_0|y|^2+\frac{b_{1,n}K_1''(x_{k_n,1})|x_{k_n,1}|}{\mu_{k_n}}\int_{\R^N}V^p\Phi_1y_1\\
&+\frac{\Delta K_2(x_{k_n})}{N\mu_{k_n}^2}b_{0,n}\int_{\R^N}U^q\Psi_0|y|^2+\frac{b_{1,n}K_2''(x_{k_n,1})|x_{k_n,1}|}{\mu_{k_n}}\int_{\R^N}U^q\Psi_1y_1+O\Big(\frac1{\mu_{k_n}^{\frac m2+1+\theta}}+\frac1{\mu_{k_n}^3}\Big)\\
=&\frac{b_{0,n}|x_{k_n,1}|}{N\mu_{k_n}^2}\Big(\int_{\R^N}V^p\Phi_0|y|^2\Big(\Delta K_1(x_{k_n})-r\Big(\Delta K_1(x_{k_n})
+\frac12\frac{\partial\Delta K_1(x_{k_n})}{\partial y_1}\Big)\Big)\\
&\qquad +\int_{\R^N}U^q\Psi_0|y|^2\Big(\Delta K_2(x_{k_n})-r\Big(\Delta K_2(x_{k_n})
+\frac12\frac{\partial\Delta K_2(x_{k_n})}{\partial y_1}\Big)\Big)
\Big)\\
&+O\Big(\frac1{\mu_{k_n}^{\frac m2+1+\theta}}+\frac1{\mu_{k_n}^3}\Big).
\end{split}
\end{align}

Since for $\dis i=1,2$, $\Delta K_i-r\big(\Delta K_i
+\frac12(\Delta K_i)'\big)\neq0$ at $r=r_0$, combining  \eqref{4.25} and \eqref{4.27},  we obtain $b_{0,n}=o(1)$, which,  in view of \eqref{4.24},   implies in turn $b_{1,n}=o(1)$.

\end{proof}

\begin{proof}[
\textbf{Proof of Theorem \ref{th3}}]

Since $(\xi_{1,n},\xi_{2,n})$ satisfies \eqref{eqxi}, we have
\begin{align}\label{eqxiint}
\begin{split}
 \xi_{1,n}(y)&=p\int_{\R^N}\frac{1}{|y-z|^{N-2}}K_1(z)v_{2,k_n}^{p-1}\xi_{2,n}(z)dz,
\\
 \xi_{2,n}(y)&=q\int_{\R^N}\frac{1}{|y-z|^{N-2}}K_2(z)v_{1,k_n}^{q-1}\xi_{1,n}(z)dz.
\end{split}\end{align}

We estimate as in the proof of Lemma \ref{lemrefine} to obtain that, for some $\theta>0$,
\begin{align*}
&\Big|\int_{\R^N}\frac{1}{|y-z|^{N-2}}K_1(z)v_{2,k_n}^{p-1}\xi_{2,n}(z)dz\Big|\\
 \leq& C\|\xi_{2,n}\|_{*,2}\int_{\R^N}\frac{1}{|y-z|^{N-2}}K_1(z)|v_{2,k_n}|^{p-1}\sum_{j=1}^{k_n}\frac{\mu_{k_n}^{\frac N{p+1}}}{(1+\mu_{k_n}|y-\bar x_j|)^{\frac{N-2}2+\tau}}\\
 \leq&C\|\xi_{2,n}\|_{*,2}\sum_{j=1}^{k_n}\frac{\mu_{k_n}^{\frac N{q+1}}}{(1+\mu_{k_n}|y-\bar x_j|)^{\frac{N-2}2+\tau+\theta}},\\
 &\Big|\int_{\R^N}\frac{1}{|y-z|^{N-2}}K_2(z)v_{1,k_n}^{q-1}\xi_{1,n}(z)dz\Big|\\
 \leq& C\|\xi_{1,n}\|_{*,1}\int_{\R^N}\frac{1}{|y-z|^{N-2}}K_2(z)|v_{1,k_n}|^{q-1}\sum_{j=1}^{k_n}\frac{\mu_{k_n}^{\frac N{q+1}}}{(1+\mu_{k_n}|y-\bar x_j|)^{\frac{N-2}2+\tau}}\\
 \leq&C\|\xi_{1,n}\|_{*,1}\sum_{j=1}^{k_n}\frac{\mu_{k_n}^{\frac N{p+1}}}{(1+\mu_{k_n}|y-\bar x_j|)^{\frac{N-2}2+\tau+\theta}}.
\end{align*}
As a result,
\begin{align*}
&\frac{| \xi_{1,n}(y)|}{\sum_{j=1}^{k_n}\frac{\mu_{k_n}^{\frac N{q+1}}}{(1+\mu_{k_n}|y-\bar x_j|)^{\frac{N-2}2+\tau}}}
\leq C\|\xi_{2,n}\|_{*,2}\frac{\sum_{j=1}^{k_n}\frac{\mu_{k_n}^{\frac N{q+1}}}{(1+\mu_{k_n}|y-\bar x_j|)^{\frac{N-2}2+\tau+\theta}}}{\sum_{j=1}^{k_n}\frac{\mu_{k_n}^{\frac N{q+1}}}{(1+\mu_{k_n}|y-\bar x_j|)^{\frac{N-2}2+\tau}}},\\
 &\frac{| \xi_{2,n}(y)|}{\sum_{j=1}^{k_n}\frac{\mu_{k_n}^{\frac N{p+1}}}{(1+\mu_{k_n}|y-\bar x_j|)^{\frac{N-2}2+\tau}}}
\leq C\|\xi_{1,n}\|_{*,1}\frac{\sum_{j=1}^{k_n}\frac{\mu_{k_n}^{\frac N{p+1}}}{(1+\mu_{k_n}|y-\bar x_j|)^{\frac{N-2}2+\tau+\theta}}}{\sum_{j=1}^{k_n}\frac{\mu_{k_n}^{\frac N{q+1}}}{(1+\mu_{k_n}|y-\bar x_j|)^{\frac{N-2}2+\tau}}}.
\end{align*}
From Lemma \ref{lemloc}, $\xi_{i,n}\rightarrow0$ in $B_{R\mu_{k_n}^{-1}}(x_{k_n,j})$ and $\sum_{i=1}^2\|\xi_{i,n}\|_{*,i}=1$, there hold that
$$\frac{| \xi_{1,n}(y)|}{\sum_{j=1}^{k_n}\frac{\mu_{k_n}^{\frac N{q+1}}}{(1+\mu_{k_n}|y-\bar x_j|)^{\frac{N-2}2+\tau}}}\ \  \ \hbox{and} \ \ \
\frac{| \xi_{2,n}(y)|}{\sum_{j=1}^{k_n}\frac{\mu_{k_n}^{\frac N{p+1}}}{(1+\mu_{k_n}|y-\bar x_j|)^{\frac{N-2}2+\tau}}}$$
attain their maximum in $\R^N\setminus\cup_{j=1}^{k_n}B_{R\mu_{k_n}^{-1}}(x_{k_n,j})$.

Therefore,
$$
\sum_{i=1}^2\|\xi_{i,n}\|_{*,i}\leq o(1)\sum_{i=1}^2\|\xi_{i,n}\|_{*,i},
$$
contradicting $\sum_{i=1}^2\|\xi_{i,n}\|_{*,i}=1$.

\end{proof}

\section*{Appendix}

\appendix

\section{Energy expansion}
\renewcommand{\theequation}{A.\arabic{equation}}
Recall that
\begin{align*}
I(u,v):=\int_{\R^N}\nabla u\cdot\nabla v
-\int_{\R^N}\Big(\frac1{p+1}K_1(\frac{|y|}{\mu})|v|^{p+1}+\frac1{q+1}K_2(\frac{|y|}{\mu})|u|^{q+1}\Big).
\end{align*}

In this section, we calculate $I(W_1,W_2)$.
\begin{Prop}\label{propa1}
\begin{align}\label{I}
\begin{split}
  I(W_1,W_2)&=k\Big(A+\frac{\bar B_1}{\lambda^{m_2}\mu^{m_2}}+\frac{\bar B_2}{\lambda^{m_1}\mu^{m_1}}+(\frac{\tilde B_2}{\lambda^{m_1-2}\mu^{m_1}}+
\frac{\tilde B_1}{\lambda^{m_2-2}\mu^{m_2}} )(\mu r_0-r)^2\\&
\quad\quad\quad-\sum_{j=2}^k\frac{B_2}{\lambda^{N-2}|x_j-x_1|^{N-2}}\Big)
+kO\Big(\frac1{\mu^{m+\theta}}+(\frac{1}{\mu^{m_1}}+\frac{1}{\mu^{m_2}})|\mu r_0-r|^3\Big),
\end{split}
\end{align}
where $r=|x_1|$, $\dis A=(1-\frac1{q+1})\int_{\R^N}U_{0,1}^{q+1}-\frac1{p+1}\int_{\R^N}V_{0,1}^{p+1}$, and $\bar B_i,\tilde B_i,B_i$ with $i=1,2$ are some positive constants.

\end{Prop}

\begin{proof}
By symmetry, we have that
\begin{align}\label{nabla}
\begin{split}
\int_{\R^N}\nabla W_1\cdot\nabla W_2&=\sum_{i=1}^k\sum_{j=1}^k\int_{\R^N}U_{x_i,\lambda}^{q}U_{x_j,\lambda}=k\Big(\int_{\R^N}U_{0,1}^{q+1}+\sum_{j=2}^k\int_{\R^N}U_{x_1,\lambda}^{q}U_{x_j,\lambda}
\Big)\\
&=k\Big(\int_{\R^N}U_{0,1}^{q+1}dy+\sum_{j=2}^k\frac{B_1}{\lambda^{N-2}|x_1-x_j|^{N-2}}+O\Big(\Big(\frac k\mu\Big)^{N-2+\theta}\Big)
\Big).
\end{split}
\end{align}

If $q\leq3$,
\begin{align*}
&\frac1{q+1} \int_{\R^N}K_2(\frac{|y|}{\mu})|W_1|^{q+1}\\
=&\frac k{q+1} \int_{\Omega_1}\Big(K_2(\frac{|y|}{\mu})\sum_{j=1}^kU_{x_j,\lambda}^{q+1}+(q+1)K_2(\frac{|y|}{\mu})\sum_{j=2}^kU_{x_1,\lambda}^{q}U_{x_j,\lambda}
+O\Big(U_{x_1,\lambda}^{\frac{q+1}{2}}(\sum_{i=2}^kU_{x_i,\lambda})^{\frac{q+1}{2}}\Big)\Big);
\end{align*}
while if $q>3$,
\begin{align*}
&\frac1{q+1} \int_{\R^N}K_2(\frac{|y|}{\mu})|W_1|^{q+1}\\
=&\frac k{q+1} \int_{\Omega_1}\Big(K_2(\frac{|y|}{\mu})U_{x_1,\lambda}^{q+1}+(q+1)K_2(\frac{|y|}{\mu})\sum_{j=2}^kU_{x_1,\lambda}^{q}U_{x_j,\lambda}\\&
\quad\quad\quad\quad\quad+O\Big(U_{x_1,\lambda}^{\frac{q+1}{2}}(\sum_{i=2}^kU_{x_i,\lambda})^{\frac{q+1}{2}}+U_{x_1,\lambda}^{q-1}(\sum_{i=2}^kU_{x_i,\lambda})^{2}\Big)\Big).
\end{align*}

Since in $\R^N\setminus\Omega_1$, $\dis |y-x_1|\geq \frac {c\mu} k$, there exists $\dis\frac{N-2-m}{N-2}<\alpha<q(N-2)-N$ (noting $\dis p,q>\frac{N+1}{N-2}$) such that $\dis q(N-2)-\theta-\alpha>N$. Then we estimate
\begin{align*}
 &\int_{\R^N\setminus\Omega_1}U_{0,1}^{q+1}+\sum_{j=2}^kU_{x_1,\lambda}^{q}U_{x_j,\lambda}\\
 \leq &C\Big(\frac k\mu\Big)^{N-2-\theta}\int_{\R^N\setminus\Omega_1}\sum_{j=2}^k\frac{1}{(1+|y-x_1|)^{q(N-2)-(N-2+\theta)}}\frac{1}{(1+|y-x_j|)^{N-2}}\\
&+ C\Big(\frac k\mu\Big)^{N-2-\theta}\int_{\R^N\setminus\Omega_1}\frac{1}{(1+|y-x_1|)^{(q+1)(N-2)-(N-2+\theta)}}\\
 \leq &C\Big(\frac k\mu\Big)^{N-2-\theta}\int_{\R^N\setminus\Omega_1}\sum_{j=2}^k\frac{1}{|x_1-x_j|^{\alpha}}(\frac{1}{(1+|y-x_1|)^{q(N-2)-\theta-\alpha}}+\frac{1}{(1+|y-x_j|)^{q(N-2)-\theta-\alpha}})\\
&+ C\Big(\frac k\mu\Big)^{N-2-\theta}\int_{\R^N\setminus\Omega_1}\frac{1}{(1+|y-x_1|)^{(q+1)(N-2)-(N-2+\theta)}}\\
=&O\Big(\big(\frac k\mu\big)^{N-2-\theta}\Big)=O\big(\mu^{-m-\theta}\big).
\end{align*}
Therefore, there hold that
\begin{align*}
 &\int_{\Omega_1}K_2(\frac{|y|}{\mu})U_{x_1,\lambda}^{q+1}\\
 =&\int_{\Omega_1}U_{x_1,\lambda}^{q+1}+\int_{\Omega_1}\Big(K_2(\frac{|y|}{\mu})-1\Big)U_{x_1,\lambda}^{q+1}
\\=&\int_{\R^N}U_{0,1}^{q+1}-\frac{c_2}{\mu^{m_1}}\int_{\Omega_1}||y-x_1|-\mu r_0|^{m_1}U_{0,1}^{q+1}
+O\big(\mu^{-m-\theta}\big)\\
=&\int_{\R^N}U_{0,1}^{q+1}-\frac{c_2}{\mu^{m_1}\lambda^{m_1}}\int_{\R^N}|y_1|^{m_1}U_{0,1}^{q+1}\\&
-\frac{c_2}{\mu^{m_1}\lambda^{m-2}}\frac{m_1(m_1-1)}2\int_{\R^N}|y_1|^{m_1-2}U_{0,1}^{q+1}(
\mu r_0-|x_1|)^2
+O\big(\mu^{-m-\theta}\big),
\end{align*}
and
\begin{align}\label{B0}
\begin{split}
\int_{\Omega_1}K_2(\frac{|y|}{\mu})\sum_{j=2}^kU_{x_1,\lambda}^{q}U_{x_j,\lambda}
&=\int_{\R^N}\sum_{j=2}^kU_{x_1,\lambda}^{q}U_{x_j,\lambda}+\Big(K_2(\frac{|y|}{\mu})-1\Big)\sum_{j=2}^kU_{x_1,\lambda}^{q}U_{x_j,\lambda}
+O\big(\mu^{-m-\theta}\big)\\
&=\sum_{j=2}^k\frac{B_1}{\lambda^{N-2}|x_1-x_j|^{N-2}}+O\big(\mu^{-m-\theta}\big).
\end{split}
\end{align}

In view of the range of $p$ and $q$, we can
take $\dis \frac{2(N-2)}{q+1}<\alpha<\min\{N-2,2(N-2)-\frac{2N}{q+1}\}=N-2$ (but $\dis\min\Big\{N-2,2(N-2)-\frac{2N}{p+1}\Big\}=2(N-2)-\frac{2N}{p+1}$) to obtain
\begin{align*}
&\int_{\Omega_1}U_{x_1,\lambda}^{\frac{q+1}{2}}\Big(\sum_{j=2}^kU_{x_j,\lambda}\Big)^{\frac{q+1}{2}}\\
\leq& \Big(\frac k\mu\Big)^{\alpha\frac{q+1}{2}}\int_{\R^N}\frac{1}{(1+|y-x_1|)^{(N-2)(q+1)+\frac{(N-2-\alpha)(q+1)}{2}}}+O\Big((\frac k{\mu})^{N-2+\theta}\Big)\\
= &O\Big((\frac k{\mu})^{N-2+\theta}\Big).
\end{align*}

Moreover, in the case  $q>3$, we  could also estimate
\begin{align*}
&\int_{\Omega_1}U_{x_1,\lambda}^{q-1}\Big(\sum_{j=2}^kU_{x_j,\lambda}\Big)^{2}\\
\leq& \Big(\frac k\mu\Big)^{2\alpha}\int_{\R^N}\frac{1}{(1+|y-x_1|)^{(N-2)(q-1)+2(N-2-\alpha)}}+O\Big((\frac k{\mu})^{N-2+\theta}\Big)\\
=& O\Big((\frac k{\mu})^{N-2+\theta}\Big).
\end{align*}

Thus,
\begin{align*}
  &\int_{\R^N}K_2(\frac{|y|}{\mu})|W_1|^{q+1}\\
=&k\Big(\int_{\R^N}U_{0,1}^{q+1}-\frac{c_2}{\mu^{m_1}\lambda^{m_1}}\int_{\R^N}|y_1|^{m_1}U_{0,1}^{q+1}\\&
\quad-\frac{c_2}{\mu^{m_1}\lambda^{m_1-2}}\int_{\R^N}\frac12 m_1(m_1-1)|y_1|^{{m_1}-2}U_{0,1}^{2^*}\big(\mu r_0-|x_1|\big)^2\\
&\quad-\sum_{j=2}^k\frac{(q+1)B_1}{\lambda^{N-2}|x_1-x_j|^{N-2}}
+O(\mu^{-m-\theta})\Big).
\end{align*}

Similar estimates hold for $W_2$,
\begin{align*}
  &\int_{\R^N}K_1(\frac{|y|}{\mu})|W_2|^{p+1}\\
=&k\Big(\int_{\R^N}V_{0,1}^{p+1}-\frac{c_1}{\mu^{m_2}\lambda^{m_2}}\int_{\R^N}|y_1|^{m_2}V_{0,1}^{p+1}\\&
\quad-\frac{c_1}{\mu^{m_2}\lambda^{m_2-2}}\int_{\R^N}\frac12 m_2(m_2-1)\int_{\R^N}|y_1|^{{m_2}-2}V_{0,1}^{2^*}\big(\mu r_0-|x_1|\big)^2\\
&\quad-\sum_{j=2}^k\frac{(p+1)B_2}{\lambda^{N-2}|x_1-x_j|^{N-2}}
+O(\mu^{-m-\theta})\Big).
\end{align*}

Combining the above estimates together, we  conclude \eqref{I}.

\end{proof}

\begin{Prop}\label{propa2}
\begin{align*}
\frac{\partial I(W_1,W_2)}{\partial\lambda}&=k\Big(-\frac{\bar B_1m_2}{\lambda^{m_2+1}\mu^{m_2}}-\frac{\bar B_2m_1}{\lambda^{m_1+1}\mu^{m_1}}+\sum_{j=2}^k\frac{B_2(N-2)}{\lambda^{N-1}|x_j-x_1|^{N-2}}\Big)\\&
\quad+kO\Big(\frac1{\mu^{m+\theta}}+(\frac{1}{\mu^{m_1}}+\frac{1}{\mu^{m_2}})|\mu r_0-r|^2\Big),
\end{align*}
where $\bar B_i,\tilde B_i,B_i$ with $i=1,2$  are the same positive constants in Proposition \ref{propa1}.

\end{Prop}
The proof of this proposition is similar to Proposition \ref{propa1} and we omit it.

\medskip
\section{Some technical estimates}\label{sb}
Now we first give the following known result which are useful in the previous sections.
\renewcommand{\theequation}{B.\arabic{equation}}
\begin{Lem}[Lemma B.1, \textbf{\cite{wei-yan-10jfa}}]\label{lemb1}
For any constant $0<\sigma\leq\min\{\alpha,\beta\}$, there is a constant $C>0$, such that
\begin{align*}
\frac{1}{(1+|y-x_i|)^\alpha}\frac{1}{(1+|y-x_j|)^\beta}
\leq\frac C{|x_i-x_j|^\sigma}\Big(\frac{1}{(1+|y-x_i|)^{\alpha+\beta-\sigma}}+\frac{1}{(1+|y-x_j|)^{\alpha+\beta-\sigma}}\Big).
\end{align*}
\end{Lem}

\medskip

Just by the similar arguments as that of Lemma B.2 in \cite{wei-yan-10jfa}, we have
\begin{Lem}\label{lemb2}
For any constant $\sigma>0,\sigma\neq N-2$, there exists a constant $C>0$, such that
\begin{align*}
\int_{\R^N}\frac 1{|y-z|^{N-2}}\frac{1}{(1+|z|)^{2+\sigma}}dz\leq \frac C{(1+|y|)^{\min\{\sigma,N-2\}}}.
\end{align*}
\end{Lem}
Finally we give the proof of \eqref{2-5}.
\begin{proof}[\textbf{Proof of \eqref{2-5}}]
\renewcommand{\theequation}{C.\arabic{equation}}
First let us recall that $\dis\bar\sigma=\frac{N-2}2+\tau$
\begin{align*}
&\left\langle L_k(\varphi_1,\varphi_2),(Z_{1,l},Y_{1,l})\right\rangle
=\left\langle L_k(Z_{1,l},Y_{1,l}),(\varphi_1,\varphi_2)\right\rangle\\
=&\int_{\R^N}p\Big(1-K_1(\frac y\mu)\Big)W_2^{p-1}Z_{1,l}\varphi_1+p\Big(V_{x_1,\lambda}^{p-1}-\big(\sum_{j=1}^kV_{x_j,\lambda}\big)^{p-1}\Big)Z_{1,l}\varphi_1dy\\
&+\int_{\R^N}q\Big(1-K_2(\frac y\mu)\Big)W_1^{q-1}Y_{1,l}\varphi_2+q\Big(U_{x_1,\lambda}^{q-1}-\big(\sum_{j=1}^kU_{x_j,\lambda}\big)^{q-1}\Big)Y_{1,l}\varphi_2dy
\end{align*}
We show that, with some $\theta>0$,
\begin{align}\label{app2-5}
\left|\left\langle L_k(\varphi_1,\varphi_2),(Y_{1,l},Z_{1,l})\right\rangle\right|=O\Big(\frac1{\mu^\theta}\Big)\|(\varphi_1,\varphi_2)\|_*.
\end{align}

In fact, it suffices  to show that
\begin{align*}
&(i)\int_{\R^N} \Big(1-K_2(\frac y\mu)\Big)W_1^{q-1}Y_{1,l}\varphi_2dy\leq  \frac C{\mu^\theta} \| \varphi_2 \|_*;\\
&(ii)\int_{\R^N} \big(U_{x_1,\lambda}^{q-1}-W_1^{q-1}\big)Y_{1,l}\varphi_2dy \leq  \frac C{\mu^\theta} \| \varphi_2 \|_*,
\end{align*}
since the counterpart terms corresponding to $W_2$ can be handled similarly. Moreover, one can refer to \cite{wei-yan-10jfa} for similar proof for (i).
We just prove (ii) in two cases respectively: $q-1\leq1$ and $q-1>1$.

First, if $q-1\leq1$, then we have
\begin{align*}
&\big|U_{x_1,\lambda}^{q-1}-W_1^{q-1}\big|\leq  \Big(\sum_{j=2}^k\frac{1}{(1+|y-x_j|)^{N-2}}\Big)^{q-1}.
\end{align*}
Since $\dis q,p>\frac{N+1}{N-2}$, we obtain $\dis (N-2)(q-1)+(N-2)+\bar\sigma>N+1+\bar\sigma$ and $\dis q(\bar\sigma-\tau_1)+N-2>\frac{3N}2-\frac32+\tau-\tau_1$ with some $\dis\tau_1>\frac{N-2-m}{N-2}$.
Thus, there exist some $\dis\alpha\in[\frac{N-2-m}{N-2},q(N-2)+\bar\sigma),\beta\in\Big[\frac{N-2-m}{N-2},q(N-2)+\bar\sigma\Big)$, and $\theta,\tilde\theta,\bar\theta>0$, such that
\begin{align*}
&\int_{\R^N} \big(U_{x_1,\lambda}^{q-1}-W_1^{q-1}\big)Y_{1,l}\varphi_2dy \\
 \leq& C \| \varphi_2 \|_*\int_{\R^N}\Big(\sum_{j=2}^k\frac{1}{(1+|y-x_j|)^{N-2}}\Big)^{q-1}\frac{1}{(1+|y-x_1|)^{N-2}}
 \sum_{j=1}^k\frac{1}{(1+|y-x_j|)^{\bar\sigma}}dy\\
 \leq& C \| \varphi_2 \|_*\Big(\int_{\R^N}\sum_{j=2}^k\frac{1}{(1+|y-x_j|)^{(N-2)(q-1)}}\frac{1}{(1+|y-x_1|)^{N-2+\bar\sigma}}dy\\
 &+\int_{\R^N}\Big(\sum_{j=2}^k\frac{1}{(1+|y-x_j|)^{\bar\sigma}}\Big)^{q}\frac{1}{(1+|y-x_1|)^{N-2}}dy\Big)\\
 \leq& C \| \varphi_2 \|_*\Big(\sum_{j=2}^k\frac{1}{|x_1-x_j|^{\alpha}}\int_{\R^N}\frac{1}{(1+|y-x_1|)^{N+\tilde\theta}}+\frac{1}{(1+|y-x_j|)^{N+\tilde\theta}}dy\\&
 +\sum_{j=2}^k\frac{1}{|x_1-x_j|^{\beta}}\int_{\R^N}\frac{1}{(1+|y-x_1|)^{N+\bar\theta}}+\frac{1}{(1+|y-x_1|)^{N+\bar\theta}}dy\Big)\\
 \leq& \frac C{\mu^\theta} \| \varphi_2 \|_*.
\end{align*}

\medskip

On the other hand, if   $q-1>1$, then we have
\begin{align*}
&\big|U_{x_1,\lambda}^{q-1}-W_1^{q-1}\big|\\
\leq& C\Big[\frac{1}{(1+|y-x_1|)^{(N-2)(q-2)}}\sum_{j=2}^k\frac{1}{(1+|y-x_j|)^{N-2}}\\
&\quad\quad+
\frac{1}{(1+|y-x_1|)^{N-2}}\big(\sum_{j=2}^k\frac{1}{(1+|y-x_j|)^{N-2}}\big)^{q-2}+(\sum_{j=2}^k\frac{1}{(1+|y-x_j|)^{N-2}})^{q-1}
\Big].
\end{align*}
Thus, following the method used when dealing with the case $p-1\leq1$, it is not difficult to show
that there exists some $\theta>0$ such that
\begin{align*}
&\int_{\R^N} \big(U_{x_1,\lambda}^{q-1}-W_1^{q-1}\big)Y_{1,l}\varphi_2dy\leq \frac C{\mu^\theta} \| \varphi_2 \|_*.
\end{align*}

Similar estimates hold to give
\begin{align*}
&\int_{\R^N} \big(V_{x_1,\lambda}^{p-1}-W_2^{p-1}\big)Z_{1,l}\varphi_1dy\leq \frac C{\mu^\theta} \| \varphi_1 \|_*.
\end{align*}

\end{proof}

\end{document}